\documentclass{amsart}
\usepackage{amssymb}
\usepackage{amsmath}
\usepackage{amsthm}
\usepackage{mathrsfs}
\usepackage{colortbl}
\newtheorem{theorem}{Theorem}
\newtheorem{lemma}[theorem]{Lemma}
\newtheorem{assumption}[theorem]{Assumption}
\newtheorem{remark}[theorem]{Remark}

\theoremstyle{definition}
\newtheorem{definition}[theorem]{Definition}

\newtheorem{proposition}[theorem]{Proposition}

\makeatletter

\@addtoreset{theorem}{section}
\makeatother

\makeatletter

\@addtoreset{equation}{section}
\makeatother

\newcommand{\dual}[1]{\langle #1 \rangle }
\begin{document}                                                 
\title[Poisson Equation on Surface with Boundary]{On the Poisson Equation on a Surface with a boundary condition in co-normal direction}                                 
\author[Hajime Koba]{Hajime Koba}                                
\address{Graduate School of Engineering Science, Osaka University,\\
1-3 Machikaneyamacho, Toyonaka, Osaka, 560-8531, Japan}                                  
\email{iti@sigmath.es.osaka-u.ac.jp}                                      
\author[Yuki Wakasugi]{Yuki Wakasugi}                                
\address{Graduate School of Engineering Science, Osaka University,\\
1-3 Machikaneyamacho, Toyonaka, Osaka, 560-8531, Japan}

\begin{abstract}
This paper considers the existence of weak and strong solutions to the Poisson equation on a surface with a boundary condition in co-normal direction. We apply the Lax-Milgram theorem and some properties of $H^1$-functions to show the existence of a unique weak solution to the surface Poisson equation when the exterior force belongs to $L_0^p$-space, where $H^1$- and $L_0^p$- functions are the ones whose value of the integral over the surface equal to zero. Moreover, we prove that the weak solution is a strong $L^p$-solution to the system. As an application, we study the solvability of ${\rm{div}_\Gamma } V = F$. The key idea of constructing a strong $L^p$-solution to the surface Poisson equation with a boundary condition in co-normal direction is to make use of solutions to the surface Poisson equation with a Dirichlet boundary condition.
\end{abstract}       
\maketitle

\section{Introduction}\label{sect1}

\begin{figure}[htbp]
\begin{center}
\input{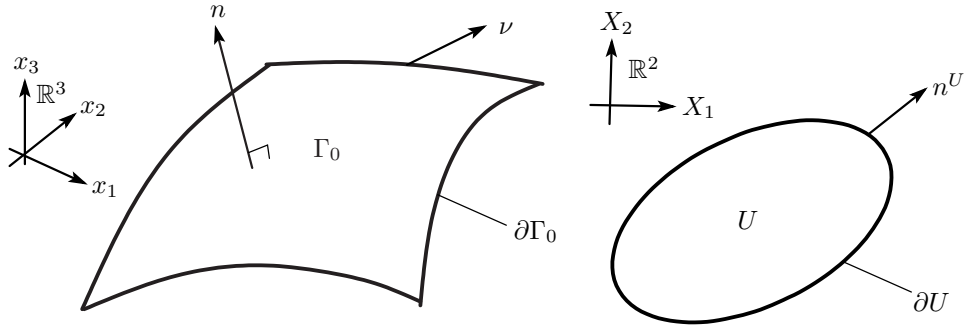}
\caption{Surface with a boundary}
\label{Fig1}
\end{center}
\end{figure}

We are interested in the existence of solutions to the Poisson equation on a surface $\Gamma_0$ with a boundary $\partial \Gamma_0$ (see Figure \ref{Fig1}). The surface Poisson equation is a simple model for steady state of concentration and thermal diffusions on the surface, for example, the spread of a substance in a soap bubble and the transmission of heat on earth's surface (see \cite{K19a} and \cite{K19b} for mathematical modeling of PDEs on surfaces with boundaries). This paper constructs weak and strong solutions to the surface Poisson equation with a boundary condition in co-normal direction. This paper has three purposes. The first one is to show the existence of a weak solution to the surface Poisson equation. The second one is to show the existence of a strong $L^p$-solution to the surface Poisson equation with a boundary condition in co-normal direction. The third one is to solve the system: ${\rm{div}}_\Gamma V = F$.

Let us first introduce basic notations. Let $x = { }^t (x_1 , x_2 , x_3 ) \in \mathbb{R}^3$, $y = { }^t (y_1 , y_2 , y_3 )$ $ \in \mathbb{R}^3$, $X = { }^t (X_1 , X_2) \in \mathbb{R}^2$, and $Y = { }^t (Y_1 ,Y_2) \in\mathbb{R}^2$ be the spatial variables. The symbols $\nabla$, $\nabla_X$, and $\nabla_Y$ are the three gradient operators defined by $\nabla = { }^t (\partial_1 , \partial_2 ,\partial_3)$, $\nabla_X = { }^t (\partial_{X_1}, \partial_{X_2})$, and $\nabla_Y = { }^t (\partial_{Y_1} , \partial_{Y_2})$, where $\partial_j = \partial / \partial x_j $, $\partial_{X_\alpha} = \partial/\partial X_\alpha$, and $\partial_{Y_\alpha} = \partial/\partial Y_\alpha$. Let $\Gamma_0$ be a surface with a boundary $\partial \Gamma_0$ such that $\Gamma_0 = \{ x \in \mathbb{R}^3;{ \ }x = \widehat{x} (X), X \in U \}$, where $U \subset \mathbb{R}^2$ and $\widehat{x} = \widehat{x} (X) = { }^t (\widehat{x}_1, \widehat{x}_2, \widehat{x}_3) \in [C^2 (\overline{U})]^3$. The symbol $n = n (x) = { }^t (n_1 , n_2 , n_3)$ denotes the unit outer normal vector at $x \in \overline{\Gamma_0}$, the symbol $\nu = \nu ( x) = { }^t ( \nu_1 , \nu_2 , \nu_3)$ denotes the unit outer co-normal vector at $x \in \partial \Gamma_0$, and the symbol $n^U = n^U (X) = { }^t (n_1^U , n_2^U)$ denotes the unit outer normal vector at $X \in \partial U$ (see Figure \ref{Fig1}).

This paper considers the \emph{surface Poisson equation} on the surface $\Gamma_0$:
\begin{equation}\label{eq11}
\begin{cases}
- \Delta_\Gamma v = F \text{ on } \Gamma_0,\\
\displaystyle{ \frac{\partial v}{\partial \nu} = 0 } \text{ on } \partial \Gamma_0,
\end{cases}
\end{equation}
and the solvability of ${\rm{div}_\Gamma } V = F$:
\begin{equation}\label{eq12}
\begin{cases}
{\rm{div}}_\Gamma V = F \text{ on } \Gamma_0,\\
V \cdot n = \chi \text{ on } \Gamma_0,\\
V \cdot \nu = 0 \text{ on }\partial \Gamma_0,
\end{cases}
\end{equation}
where $v = v (x)$, $V = V (x) = { }^t ( V_1 , V_2 , V_3)$ are unknown functions and the exterior forces $F = F (x)$, $\chi = \chi (x)$ are given functions. Here $\Delta_\Gamma v := (\partial_1^\Gamma)^2 v + ( \partial_2^\Gamma )^2 v + (\partial_3^\Gamma )^2 v$, ${\partial v }/{\partial \nu} := ( \nu \cdot \nabla_\Gamma ) v $, and ${\rm{div}}_\Gamma V := \nabla_\Gamma \cdot V$, where $\nabla_\Gamma := { }^t (\partial_1^\Gamma , \partial_2^\Gamma , \partial_3^\Gamma )$ and $\partial_j^\Gamma := \partial_j - n_j (n \cdot \nabla )$. See Sections \ref{sect2} and \ref{sect3} for the differential operators $\Delta_\Gamma$, $\nabla_\Gamma$, $\partial_j^\Gamma$, and the unit outer co-normal vector $\nu$. System \eqref{eq12} plays an important role in studying the surface pressure of the fluid on surfaces (see \cite{KLG17}). This paper applies a solution to system \eqref{eq11} to construct a solution to \eqref{eq12}.

Let us state mathematical analysis of the surface Poisson equation: $- \Delta_\Gamma v = F$ on $\Gamma_0$. Aubin \cite[Chapter 4]{Aub82} studied system \eqref{eq11} when $\Gamma_0$ is a compact $C^\infty$ Riemannian manifold. They showed the existence of a smooth solution to system \eqref{eq11} if $F$ is a smooth function. They also showed the existence of the Green function of the Laplacian $\Delta_\Gamma$ when $\Gamma_0$ is the oriented compact Riemannian manifold with boundary of class $C^\infty$. Ni-Shi-Tam \cite{NST01} studied the sufficient and necessary conditions for the existence of solutions to system \eqref{eq11} when $\Gamma_0$ is a complete noncompact manifolds with nonnegative Ricci curvature and $f$ is nonnegative locally H$\ddot{\rm{o}}$lder continuous function. Ni \cite[\S 2]{Ni02} studied the solutions to system \eqref{eq11} when $\Gamma_0$ is a complete Riemannian manifold. They showed the existence of a nonnegative solution to their system if $F$ is a nonnegative continuous function. See also Taylor \cite[Chapter 5, \S 9]{Tay11a} for the regularity of solutions to the natural boundary problems for the Hodge Laplacian. When $\Gamma_0$ is a complete noncompact manifolds with nonnegative Ricci curvature, Wong-Zhang \cite{WZ03}, Munteanu-Sesum \cite{MS10}, and Catino-Monticelli-Punzo \cite{CMP19} studied the existence and decay properties of solutions to their surface Poisson equations. Finally, we introduce the results related to this paper. Li-Shi-Sun \cite{LSS17} considered their approximate solutions to the Neumann problem of the surface Poisson equation under the assumptions that their Poisson equation admits a smooth solution. In this paper, we study the existence of both weak and strong solutions to the surface Poisson equation with a boundary condition in co-normal direction.

To study system \eqref{eq11}, we set $\widehat{v} = \widehat{v} (X) = v (\widehat{x} (X)) $ and $\widehat{F} = \widehat{F} (X) = \widehat{F} (\widehat{x} (X))$. Then we have
\begin{equation}\label{eq13}
\begin{cases}
- {\displaystyle{\sum_{\alpha, \beta =1}^2 \frac{1}{\sqrt{ \mathcal{G} }} \frac{\partial}{\partial X_\alpha} \left( \sqrt{ \mathcal{G} } \mathfrak{g}^{\alpha \beta} \frac{\partial \widehat{v}}{\partial X_\beta} \right) = \widehat{F}}} \text{ in } U,\\
{\displaystyle{ \frac{ n_1^U \mathfrak{g}_2 - n_2^U \mathfrak{g}_1}{| n_1^U \mathfrak{g}_2 - n_2^U \mathfrak{g}_1| } \times \frac{ \mathfrak{g}_1 \times \mathfrak{g}_2}{| \mathfrak{g}_1 \times \mathfrak{g}_2| }}} \cdot {\displaystyle{\sum_{\alpha , \beta =1}^2 \mathfrak{g}^{\alpha \beta} \mathfrak{g}_\alpha \frac{\partial \widehat{v}}{\partial X_\beta} }} = 0 \text{ on } \partial U,
\end{cases}
\end{equation}
where $\mathfrak{g}_\alpha := \partial \widehat{x}/{\partial X_\alpha}$, $\mathcal{G}:= \mathfrak{g}_{11} \mathfrak{g}_{22} - \mathfrak{g}_{12} \mathfrak{g}_{21}$, $\mathfrak{g}_{\alpha \beta} := \mathfrak{g}_\alpha \cdot \mathfrak{g}_\beta$, and $( \mathfrak{g}^{\alpha \beta})_{2 \times 2} := (\mathfrak{g}_{\alpha \beta})_{2 \times 2}^{-1}$. See Section \ref{sect2} for details.

To show the existence of a strong $L^p$-solution to system \eqref{eq11}, we have to find a function $\widehat{v} \in W^{2,p} (U)$ satisfying
\begin{equation}\label{eq14}
{\displaystyle{ \frac{ n_1^U \mathfrak{g}_2 - n_2^U \mathfrak{g}_1}{| n_1^U \mathfrak{g}_2 - n_2^U \mathfrak{g}_1| } \times \frac{ \mathfrak{g}_1 \times \mathfrak{g}_2}{| \mathfrak{g}_1 \times \mathfrak{g}_2| }}} \cdot {\displaystyle{\sum_{\alpha , \beta =1}^2 \mathfrak{g}^{\alpha \beta} \mathfrak{g}_\alpha \frac{\partial \widehat{v}}{\partial X_\beta} }} = 0 \text{ on } \partial U.
\end{equation}
However, it is not easy to deal with \eqref{eq14} directly. To overcome the difficulty, we first show the existence of a weak solution to system \eqref{eq11}, and then we prove that the weak solution is a strong $L^p$-solution to the surface Poisson equation with the boundary condition $\partial v/{\partial \nu}|_{\partial \Gamma_0} = 0$.

Let us explain the two key ideas of constructing strong solutions to system \eqref{eq11}. The first one is to apply some properties of the function spaces $L_0^p (\Gamma_0)$ and $H^1 (\Gamma_0)$ defined by
\begin{align*}
L_0^p (\Gamma_0) & = \left\{ f \in L^p (\Gamma_0 ); { \ } \int_{\Gamma_0} f (y) { \ }d \mathcal{H}_y^2 = 0 \right\},\\
H^1 (\Gamma_0) & = \left\{ f \in W^{1,2} (\Gamma_0 ); { \ } \int_{\Gamma_0} f (y) { \ }d \mathcal{H}_y^2 = 0 \right\},
\end{align*}
where $d \mathcal{H}^2_y$ denotes the 2-dimensional Hausdorff measure (see Section \ref{sect3} for our function spaces on the surface $\Gamma_0$). We apply both Lax-Milgram theorem and the surface Poincar\'e inequality for $H^1 (\Gamma_0)$-functions to show the existence of a unique weak solution to system \eqref{eq11} when $F \in L_0^2 (\Gamma_0)$.

The second one is to consider the following surface Poisson equation with the Dirichlet boundary condition:
\begin{equation}\label{eq15}
\begin{cases}
- \Delta_\Gamma v_* & = F \text{ on } \Gamma_0,\\
v_* \big|_{\partial \Gamma_0} & = v|_{\partial \Gamma_0},
\end{cases}
\end{equation}
where $v$ is a weak solution to system \eqref{eq11} when $F \in L_0^2 (\Gamma_0)$. By constructing a strong solution to system \eqref{eq15}, we show the existence of a strong solution to \eqref{eq11}. See Sections \ref{sect4}-\ref{sect6} for details.

The outline of this paper is as follows: In Section \ref{sect2}, we state the assumptions of our surface $\Gamma_0$ and the main results of this paper. In Section \ref{sect3}, we study function spaces on the surface $\Gamma_0$. In subsection \ref{subsec31}, we recall the representation of differential operators on the surface $\Gamma_0$ and the surface divergence theorem. In subsection \ref{subsec32}, we introduce and study the function spaces $L^p (\Gamma_0)$, $W^{1,p}(\Gamma_0 )$, $W^{2,p} (\Gamma_0)$, and the trace operator $\gamma_p: W^{1,p} ( \Gamma_0 ) \to L^p (\partial \Gamma_0)$. In subsection \ref{subsec33}, we investigate the function spaces $L_0^p (\Gamma_0)$ and $H^1 (\Gamma_0)$, and derive the surface Poincar\'{e} inequalities. In Section \ref{sect4}, we apply the Lax-Milgram theorem to show the existence of a unique weak solution to system \eqref{eq11} when $F \in L_0^2 ( \Gamma_0 )$. In Section \ref{sect5} we study the regularity for weak solutions to the surface Poisson equation with a Dirichlet boundary condition. In Section \ref{sect6}, we show the existence of a unique strong $L^p$-solution to system \eqref{eq11} when $F \in L_0^p ( \Gamma_0 )$ by making use of the regularity for a solution to system \eqref{eq15}. Moreover, we apply a solution to \eqref{eq11} to construct a solution to system \eqref{eq12}.

\section{Main Results}\label{sect2}
We first introduce the definition of our surface with a boundary and notations. Then we state the main results of this paper.
\begin{definition}[Surface with boundary]\label{def21}
Let $\Gamma_0 \subset \mathbb{R}^3$ be a set. We call $\Gamma_0$ a \emph{surface with a boundary} if the following properties hold:\\
$(\mathrm{i})$ The set $\Gamma_0$ can be written as
\begin{equation*}
\Gamma_0 = \{ x = { }^t (x_1, x_2 ,x_3) \in \mathbb{R}^3; { \ }x = \widehat{x} (X), X \in U \},
\end{equation*}
where $U \subset \mathbb{R}^2$ is a bounded domain with a $C^2$-boundary $\partial U$ and $\widehat{x} = { }^t ( \widehat{x}_1 , \widehat{x}_2, \widehat{x}_3 ) \in [C^2 ( \overline{U})]^3$.\\
$(\mathrm{ii})$ The map
\begin{equation*}
\widehat{x} (\cdot ) : \overline{U} \to \overline{\Gamma_0} \text{ is bijective}.
\end{equation*}
$(\mathrm{iii})$ The boundary $\partial \Gamma_0$ of $\Gamma_0$ can be written as
\begin{equation*}
\partial \Gamma_0 = \{ x = { }^t (x_1, x_2 ,x_3) \in \mathbb{R}^3; { \ }x = \widehat{x} (X), X \in \partial U \}.
\end{equation*}
$(\mathrm{iv})$ There is $\lambda_{min} >0$ such that for all $X \in \overline{U}$
\begin{multline*}
\left( \frac{\partial \widehat{x}_2}{\partial X_1} \frac{\partial \widehat{x}_3}{\partial X_2} - \frac{\partial \widehat{x}_2}{\partial X_2} \frac{\partial \widehat{x}_3}{\partial X_1} \right)^2 + \left( \frac{\partial \widehat{x}_1}{\partial X_1} \frac{\partial \widehat{x}_3}{\partial X_2} - \frac{\partial \widehat{x}_1}{\partial X_2} \frac{\partial \widehat{x}_3}{\partial X_1} \right)^2\\
 + \left( \frac{\partial \widehat{x}_1}{\partial X_1} \frac{\partial \widehat{x}_2}{\partial X_2} - \frac{\partial \widehat{x}_1}{\partial X_2} \frac{\partial \widehat{x}_2}{\partial X_1} \right)^2 \geq \lambda_{min}^2 .
\end{multline*}
\end{definition}

Let us explain the conventions used in this paper. We use the Greek characters $\alpha$, $\beta$, $\alpha'$, $\beta'$ as $1, 2$, that is, $\alpha, \beta, \alpha' , \beta' \in \{1 , 2 \}$. Moreover, we often use the following Einstein summation convention: $c_\alpha d_{\alpha \beta} = \sum_{\alpha = 1}^2 c_\alpha d_{\alpha \beta}$ and $c_{\beta'} d^{\alpha \beta'} = \sum_{\beta' = 1}^2 c_\alpha d_{\alpha \beta'}$. The symbol $d \mathcal{H}^k_x$ denotes the $k$-dimensional Hausdorff measure.

Next we define some notations. Let $\Gamma_0$ be a surface with a boundary $\partial \Gamma_0$. By definition, there are bounded domain $U \subset \mathbb{R}^2$ with a $C^2$-boundary $\partial U$ and $\widehat{x} = \widehat{x} (X) = { }^t ( \widehat{x}_1 , \widehat{x}_2 , \widehat{x}_3 ) \in [ C^3 ( \overline{U}) ]^3$ satisfying the properties as in Definition \ref{def21}. The symbol $n = n (x) = { }^t (n_1 , n_2 , n_3)$ denotes the unit outer normal vector at $x \in \overline{\Gamma_0}$, the symbol $\nu = \nu ( x) = { }^t ( \nu_1 , \nu_2 , \nu_3)$ denotes the unit outer co-normal vector at $x \in \partial \Gamma_0$, and the symbol $n^U = n^U (X) = { }^t (n_1^U , n_2^U)$ denotes the unit outer normal vector at $X \in \partial U$.

Fix $j \in \{ 1 , 2 , 3 \}$. For each $\psi \in C^1 ( \mathbb{R}^3)$, ${\bf{f}} = { }^t ( {\bf{f}}_1 , {\bf{f}}_2 , {\bf{f}}_3) \in [ C^1 ( \mathbb{R}^3)]^3$, and $f \in C^2 ( \mathbb{R}^3)$, ${\displaystyle{\partial_j^\Gamma \psi := \sum_{i=1}^3 ( \delta_{i j} - n_i n_j ) \partial_i \psi}}$, $\nabla_\Gamma := { }^t (\partial_1^\Gamma , \partial_2^\Gamma , \partial_3^\Gamma )$, ${\rm{div}}_\Gamma {\bf{f}} : = \nabla_\Gamma \cdot {\bf{f}} = \partial_1^\Gamma {\bf{f}}_1 + \partial_2^\Gamma {\bf{f}}_2 + \partial_3^\Gamma {\bf{f}}_3$, $\Delta_\Gamma f  := (\partial_1^\Gamma)^2 f  + (\partial_2^\Gamma)^2 f  + (\partial_3^\Gamma)^2 f $. Here $\delta_{ij}$ denotes the Kronecker delta.

We state the definitions of weak and strong solutions to system \eqref{eq11}.
\begin{definition}[Weak solutions]\label{def22}
Let $F \in L_0^2 ( \Gamma_0 )$ and $v \in H^1 (\Gamma_0)$. We call $v$ a \emph{weak solution} to system \eqref{eq11} if 
\begin{equation*}
\dual{ \nabla_\Gamma v , \nabla_\Gamma \psi } = \dual{F , \psi }
\end{equation*}
holds for all $\psi \in W^{1, 2} ( \Gamma_0 )$. Here $\dual{\cdot, \cdot}$ denotes the $L^2$-inner product defined by \eqref{eq37}.
\end{definition}

\begin{definition}[Strong solutions]\label{def23}
Let $2 \leq p < \infty$, and let $F \in L_0^p ( \Gamma_0 )$ and $v \in W^{2,p} (\Gamma_0)$. We call $v$ a \emph{strong $L^p$-solution} to system \eqref{eq11} if 
\begin{equation*}
\|  \Delta_\Gamma v + F \|_{L^p (\Gamma_0)} = 0 \text{ and  }\left\| \gamma_p [{\partial v}/{\partial \nu}] \right\|_{L^p (\partial \Gamma_0 )} = 0. 
\end{equation*}
Here $\gamma_p : W^{1,p} (\Gamma_0) \to L^p (\partial \Gamma_0)$ is the trace operator defined by Definition \ref{Def37}.
\end{definition}
\noindent See Section \ref{sect3} for the function spaces $L_0^p (\Gamma_0)$, $H^1 ( \Gamma_0)$, $W^{1,p} (\Gamma_0 )$, and $W^{2,p}(\Gamma_0)$.

Before stating the main results of this paper, we define some notations. For every $X \in \overline{U}$,
\begin{multline*}
\mathfrak{g}_1 = \mathfrak{g}_1 (X) := \frac{\partial \widehat{x}}{\partial X_1},{ \ }\mathfrak{g}_2 = \mathfrak{g}_2 (X) := \frac{\partial \widehat{x}}{\partial X_2}, { \ } \mathfrak{g}_{\alpha \beta} := \mathfrak{g}_\alpha \cdot \mathfrak{g}_\beta,\\
(\mathfrak{g}^{\alpha \beta})_{2 \times 2} : = ( \mathfrak{g}_{\alpha \beta})_{2\times 2}^{-1} = \frac{1}{ \mathfrak{g}_{11} \mathfrak{g}_{22} - \mathfrak{g}_{12} \mathfrak{g}_{21}}
\begin{pmatrix}
\mathfrak{g}_{22} & - \mathfrak{g}_{21}\\
- \mathfrak{g}_{12} & \mathfrak{g}_{22}
\end{pmatrix}, { \ }\mathfrak{g}^\alpha := \mathfrak{g}^{\alpha \beta} \mathfrak{g}_\beta,\\
\mathcal{G} = \mathcal{G} (X) := \mathfrak{g}_{11} \mathfrak{g}_{22} - \mathfrak{g}_{12} \mathfrak{g}_{21}.
\end{multline*}
Note that $\sqrt{ \mathcal{G}} = \sqrt{ | \mathfrak{g}_1 \times \mathfrak{g}_2|} = \lambda_{min} > 0$ from Definition \ref{def21}.

\begin{assumption}\label{ass24} There is $\lambda_0 >0$ such that for all $X \in \overline{U}$ and $\xi = { }^t ( \xi_1 , \xi_2 ) \in \mathbb{R}^2$
\begin{equation*}
\mathfrak{g}^{\alpha \beta} \xi_\alpha \xi_\beta \geq \lambda_0 (\xi_1^2 + \xi_2^2).
\end{equation*}
\end{assumption}
Throughout this paper, we assume that $\Gamma_0$ satisfies Assumption \ref{ass24}. We now state the main results of this paper.
\begin{theorem}[Strong $L^2$-solution]\label{thm25}
Let $F \in L_0^2 ( \Gamma_0 )$. Then there exists a unique strong $L^2$-solution $v$ in $H^1 (\Gamma_0) \cap W^{2,2} (\Gamma_0 )$ to system \eqref{eq11}.
\end{theorem}
\begin{theorem}[Strong $L^p$- and classical solutions]\label{thm26}
Let $2 < p < \infty$ and $F \in L_0^p ( \Gamma_0 )$. Then there exists a unique strong $L^p$-solution $v$ in $L_0^p(\Gamma_0 ) \cap W^{2,p} (\Gamma_0 ) \cap C^{1, 1-2/p}(\overline{\Gamma_0})$ to system \eqref{eq11}. Moreover, assume in addition that $\partial U$ is $C^3$-class, $\widehat{x} \in [C^3 ( \overline{U})]^3$, and that $F \in W^{1,p} (\Gamma_0)$. Then $v \in C^{2, 1 - 2/p } (\overline{\Gamma_0})$. 
\end{theorem}

\begin{theorem}[Solvability of ${\rm{div}_\Gamma } V = F$]\label{thm27}$(\mathrm{i})$ Let $2 \leq p < \infty$, and let $F \in L_0^p ( \Gamma_0 )$ and $\chi \in W^{1,p}(\Gamma_0)$. Then there exists $V$ in $[ W^{1,p} (\Gamma_0 )]^3$ satisfying $\| {\rm{div}}_\Gamma V - F \|_{L^p (\Gamma_0)} = 0$, $\| V \cdot n - \chi \|_{L^p (\Gamma_0)}=0$, and $\| \gamma_p (V \cdot \nu ) \|_{L^p(\partial \Gamma_0)} = 0$. Here $\gamma_p : W^{1,p} (\Gamma_0) \to L^p (\partial \Gamma_0)$ is the trace operator defined by Definition \ref{Def37}.\\
$(\mathrm{ii})$ Assume that $\partial U$ is $C^3$-class and $\widehat{x} \in [ C^3 ( \overline{U})]^3$. Let $2 < p < \infty$, and let $F \in L_0^p ( \Gamma_0) \cap W^{1,p} (\Gamma_0)$ and $\chi \in W^{2,p} (\Gamma_0)$. Then there is $V \in [C^1 (\overline{\Gamma_0})]^3$ satisfying system \eqref{eq12}.
\end{theorem}

To prove Theorems \ref{thm25}-\ref{thm27}, we prove the following two key propositions.
\begin{proposition}[Weak solution]\label{prop28}
Let $F \in L_0^2 ( \Gamma_0 )$. Then there exists a unique weak solution $v$ in $H^1 (\Gamma_0)$ to system \eqref{eq11}. Moreover, there is $C_\dagger = C_\dagger ( \Gamma_0 ) >0$ such that
\begin{equation*}
\| v \|_{W^{1,2} ( \Gamma_0) } \leq C_\dagger \| F \|_{L^2 ( \Gamma_0 )}.
\end{equation*}
\end{proposition}

\begin{proposition}[Regularity for weak solutions]\label{prop29}
Let $F \in L^2 (\Gamma_0 )$ and $w \in W_0^{1,2} (\Gamma_0)$. Assume that 
\begin{equation*}
\dual{\nabla_\Gamma w , \nabla_\Gamma \phi } = \dual{F , \phi}
\end{equation*}
holds for all $\phi \in W_0^{1,2} (\Gamma_0)$. Then $w \in W^{2,2} (\Gamma_0)$ and there is $C_{\ddagger} = C_{\ddagger} (\Gamma_0) >0$ such that
\begin{equation*}
\| w \|_{W^{2,2} ( \Gamma_0 )} \leq C_{\ddagger} \| F \|_{L^2 ( \Gamma_0 )}.
\end{equation*}
\end{proposition}
In Section \ref{sect4}, we prove Proposition \ref{prop28} to show the existence of a weak solution to system \eqref{eq11}. In Section \ref{sect5}, we derive Proposition \ref{prop29} to study the regularity for weak solutions to the surface Poisson system. In Section \ref{sect6}, we prove Theorems \ref{thm25}-\ref{thm27} to show the existence of strong solutions to our equations.

\section{Function spaces on a surface}\label{sect3}
In this section, we study some function spaces on the surface $\Gamma_0$. In subsection \ref{subsec31}, we recall the representation formula for differential operators on the surface $\Gamma_0$ and the surface divergence theorem. In subsection \ref{subsec32}, we introduce the function spaces $L^p (\Gamma_0)$, $W^{1,p} (\Gamma_0)$, $W^{2,p}(\Gamma_0)$, and the trace operator $\gamma_p : W^{1,p} ( \Gamma_0 ) \to L^p (\partial \Gamma_0)$. In subsection \ref{subsec33}, we investigate the function spaces $L_0^p (\Gamma_0)$ and $H^1 (\Gamma_0)$, and derive the surface Poincar\'{e} inequalities.

Let $\breve{X} = \breve{X} (x)$ be the inverse mapping of $\widehat{x} = \widehat{x}(X)$, i.e. $\breve{X}: \overline{\Gamma_0} \to \overline{U}$ and for all $X \in \overline{U}$ and $x \in \overline{\Gamma_0}$, $\breve{X}(\widehat{x}(X)) =X$ and $\widehat{x}(\breve{X} (x)) = x$. Throughout this paper, for $\psi = \psi (x)$ and $\varphi = \varphi (X)$,
\begin{equation*}
\widehat{ \psi } = \widehat{ \psi } (X) := \psi ( \widehat{x} (X) ) \text{ and } \breve{\varphi} = \breve{\varphi}(x) := \varphi (\breve{X}(x)) .
\end{equation*}

\subsection{Differential operators on a surface}\label{subsec31}
Let us recall the representation formula for differential operators on the surface $\Gamma_0$ and the surface divergence theorem. From \cite[Chapter 3]{Jos11}, \cite[Appendix]{DE07}, \cite[Section 3]{K18}, and \cite[Section 3]{K19a}, we obtain the following lemma.
\begin{lemma}[Representation formula for differential operators]\label{Lem31}{ \ }\\
$(\mathrm{i})$ For each $\psi \in C (\mathbb{R}^3)$,
\begin{equation}\label{eq31}
\int_{\Gamma_0} \psi (x) { \ }d \mathcal{H}^2_x = \int_U \widehat{\psi} \sqrt{\mathcal{G}} { \ }d X.
\end{equation}
$(\mathrm{ii})$ For each $j=1,2,3,$ and $\psi \in C^1 (\mathbb{R}^3)$,
\begin{equation}\label{eq32}
\int_{\Gamma_0} \partial_j^\Gamma \psi (x) { \ }d \mathcal{H}^2_x = \int_U \mathfrak{g}^{\alpha \beta} \frac{\partial \widehat{x}_j}{\partial X_\alpha} \frac{\partial \widehat{\psi}}{\partial X_\beta} \sqrt{\mathcal{G}} { \ }d X.
\end{equation}
$(\mathrm{iii})$ For each $i,j=1,2,3$, and $\psi \in C^2 (\mathbb{R}^3)$,
\begin{equation}\label{eq33}
\int_{\Gamma_0} \partial_i^\Gamma \partial_j^\Gamma \psi { \ }d \mathcal{H}^2_x =  \int_U \mathfrak{g}^{\alpha' \beta'} \frac{\partial \widehat{x}_i}{\partial X_{\alpha'} } \frac{\partial}{\partial X_{\beta'}} \left( \mathfrak{g}^{\alpha \beta} \frac{\partial \widehat{x}_j}{\partial X_\alpha} \frac{\partial \widehat{\psi} }{\partial X_\beta} \right) \sqrt{ \mathcal{G} } { \ }d X.
\end{equation}
$(\mathrm{iv})$ For each $\psi \in C^2 (\mathbb{R}^3)$,
\begin{equation}\label{eq34}
\int_{\Gamma_0} \Delta_\Gamma \psi (x) { \ }d \mathcal{H}^2_x = \int_U \left\{ \frac{1}{\sqrt{\mathcal{G}}} \frac{\partial}{\partial X_\alpha} \left( \sqrt{\mathcal{G}} \mathfrak{g}^{\alpha \beta} \frac{\partial \widehat{\psi} }{\partial X_\beta} \right) \right\} \sqrt{\mathcal{G}} { \ }d X.
\end{equation}
$(\mathrm{v})$ For each $\psi_\sharp ,\psi_\flat \in C^1 (\mathbb{R}^3)$,
\begin{equation}\label{eq35}
\int_{\Gamma_0} \nabla_\Gamma \psi_\sharp \cdot \nabla_\Gamma \psi_\flat (x) { \ }d \mathcal{H}^2_x = \int_U \mathfrak{g}^{\alpha \beta} \frac{\partial \widehat{\psi}_\sharp}{\partial X_\alpha} \frac{\partial \widehat{\psi}_\flat}{\partial X_\beta} \sqrt{\mathcal{G}} { \ }d X.
\end{equation}
\end{lemma}
From \cite[Chapter 2]{Sim83} and \cite[Section 3]{K19b}, we have
\begin{lemma}[Surface divergence theorem]\label{Lem32}For every ${\bf{f}} = { }^t ( {\bf{f}}_1 , {\bf{f}}_2 , {\bf{f}}_3 ) \in [C^1 ( \mathbb{R}^3 )]^3$,
\begin{equation*}
\int_{\Gamma_0} {\rm{div}}_\Gamma {\bf{f}} { \ }d \mathcal{H}^2_x = - \int_{\Gamma_0} H_\Gamma (n \cdot {\bf{f}}) { \ } d \mathcal{H}^2_x + \int_{\partial \Gamma_0} \nu \cdot {\bf{f}} { \ } d \mathcal{H}^1_x,
\end{equation*}
where $\rm{div}_\Gamma {\bf{f}} = \partial_1^\Gamma {\bf{f}}_1 + \partial_2^\Gamma {\bf{f}}_2 + \partial_3^\Gamma {\bf{f}}_3$ and  $H_\Gamma = H_\Gamma (x )$ denotes the mean curvature in the direction $n$ defined by $H_\Gamma = - {\rm{div}}_\Gamma n$. 
\end{lemma}
From \cite[Section 3]{K19b} we find that the unit outer co-normal vector $\nu$ is represented by
\begin{equation*}
\nu ( \widehat{x} (X) ) = \frac{n_1^U \mathfrak{g}_2 - n_2^U \mathfrak{g}_1}{ | n_1^U \mathfrak{g}_2 -n_2^U \mathfrak{g}_1| } \times \frac{\mathfrak{g}_1 \times \mathfrak{g}_2}{ | \mathfrak{g}_1 \times \mathfrak{g}_2 | } \text{ for a.e. } X \in \partial U.
\end{equation*}

\subsection{Function spaces on a surface}\label{subsec32}
Let us define and study basic function spaces on the surface $\Gamma_0$. For $k=0,1,2$, $0 < \kappa <1$, and $1 \leq p < \infty$,
\begin{align*}
C^k ( \Gamma_0) & := \{ \psi : \Gamma_0 \to \mathbb{R} ;{ \ }\psi = \breve{\varphi}, { \ }\varphi \in C^k ( U )\},\\
C^k ( \overline{\Gamma_0}) & := \{ \psi : \overline{\Gamma_0} \to \mathbb{R} ;{ \ }\psi = \breve{\varphi}, { \ }\varphi \in C^k (\overline{U})\},\\
C^k_0 ( \Gamma_0) & := \{ \psi : \Gamma_0 \to \mathbb{R} ;{ \ }\psi = \breve{\varphi}, { \ }\varphi \in  C_0^k(U)\},\\
C^{k,\kappa} ( \overline{\Gamma_0}) & := \{ \psi : \overline{\Gamma_0} \to \mathbb{R} ;{ \ }\psi = \breve{\varphi}, { \ }\varphi \in C^{k,\kappa} (\overline{U})\},\\
L^p ( \Gamma_0) & := \{ \psi : \Gamma_0 \to \mathbb{R} ;{ \ }\psi = \breve{\varphi}, { \ }\varphi \in L^p(U),{ \ }\| \psi \|_{L^p (\Gamma_0)} < \infty \},\\
L^{\infty} ( \Gamma_0) & := \{ \psi : \Gamma_0 \to \mathbb{R} ;{ \ }\psi = \breve{\varphi}, { \ }\varphi \in L^\infty (U),{ \ }\| \psi \|_{L^\infty (\Gamma_0)} < \infty \}.
\end{align*}
Here
\begin{align}
\| \psi \|_{L^p( \Gamma_0)} & := \left( \int_{\Gamma_0} | \breve{\varphi} (x) |^p { \ }d \mathcal{H}^2_x \right)^{\frac{1}{p}},\label{eq36}\\
\| \psi \|_{L^\infty ( \Gamma_0)} & : = \text{ess.sup}_{x \in \Gamma_0} | \breve{\varphi} (x)|.\notag
\end{align}
Moreover, for $\psi \in C ( \overline{\Gamma_0}) (= C^0 (\overline{\Gamma_0}))$,
\begin{equation*}
\| \psi |_{\partial \Gamma_0 } \|_{L^p ( \partial \Gamma_0 ) } := \left( \int_{\partial \Gamma_0} | { \ }\breve{\varphi}|_{\partial \Gamma_0} { \ } |^p { \ } d \mathcal{H}_x^1 \right)^{\frac{1}{p} },
\end{equation*}
and for every $\psi_\sharp \in L^p (\Gamma_0)$ and $\psi_\flat \in L^{p'} (\Gamma_0 )$
\begin{equation}\label{eq37}
\dual{\psi_\sharp , \psi_\flat} := \int_{\Gamma_0} \psi_\sharp (x) \psi_\flat (x) { \ }d \mathcal{H}^2_x,
\end{equation}
where $1 < p' \leq \infty$ such that $1/p + 1/{p'} = 1$. Define the inner product of $L^2 (\Gamma_0)$ by \eqref{eq37}.

\begin{remark}\label{Rem33}
Since $0 < \lambda_{min} \leq \sqrt{ \mathcal{G}} < + \infty$, it follows from \eqref{eq31} to see that there is $C>0$ such that for each $\psi \in L^p(\Gamma_0)$ and $\varphi \in L^p (U)$ such that $\psi = \breve{\varphi}$,
\begin{equation}\label{eq38}
\lambda_{min}^{\frac{1}{p}} \| \varphi \|_{L^p (U)} \leq \| \psi \|_{L^p (\Gamma_0)} \leq \lambda_{max}^{\frac{1}{p}} \| \varphi \|_{L^p (U)}.
\end{equation}
Here
\begin{equation*}
\lambda_{max} := \max_{X \in \overline{U}} | \sqrt{ \mathcal{G} }|.
\end{equation*}
We also see that for each $\psi_\sharp \in L^p (\Gamma_0)$ and $\psi_\flat \in L^{p'} (\Gamma_0 )$
\begin{equation}\label{eq39}
| \dual{\psi_\sharp , \psi_\flat} | \leq \| \psi_\sharp \|_{L^p(\Gamma_0 )} \| \psi_\flat \|_{L^{p'} (\Gamma_0 )},
\end{equation}
where $1 \leq p , p' \leq \infty$ such that $1/p + 1/{p'} = 1$.
\end{remark}

Next we define a weak derivative for functions on the surface $\Gamma_0$. For $\psi \in C^k ( \overline{\Gamma_0})$ or $\psi \in C_0^k ( \Gamma_0 )$, we define the differential operators $\partial_j^\Gamma$ and $\partial_i^\Gamma \partial_j^\Gamma $ as in Lemma \ref{Lem31}.
\begin{definition}[Weak derivatives]\label{Def34}
Let $1 \leq p \leq \infty$, $\psi \in L^p ( \Gamma_0 )$, and $j=1,2,3$. We say that $\partial^\Gamma_j \psi \in L^p (\Gamma_0 )$ if there exists $\Upsilon \in L^p ( \Gamma_0)$ such that for all $\phi \in C_0^1 ( \Gamma_0)$,
\begin{equation*}
\int_{\Gamma_0} \Upsilon \phi { \ }d \mathcal{H}^2_x = - \int_{\Gamma_0} \psi (\partial_j^\Gamma \phi + H_\Gamma n_j \phi ) { \ }d \mathcal{H}^2_x.
\end{equation*}
In particular, we write $\Upsilon$ as $\partial_j^\Gamma \psi$.
\end{definition}

Now we introduce Sobolev spaces on the surface $\Gamma_0$. For $1 \leq p < \infty$,
\begin{align*}
W^{1,p} ( \Gamma_0) & := \{ \psi : \Gamma_0 \to \mathbb{R} ;{ \ }\psi = \breve{\varphi}, { \ }\varphi \in W^{1,p}(U),{ \ }\| \psi \|_{W^{1,p} (\Gamma_0)} < \infty \},\\
W_0^{1,p} ( \Gamma_0) & := \{ \psi : \Gamma_0 \to \mathbb{R} ;{ \ }\psi = \breve{\varphi}, { \ }\varphi \in W_0^{1,p}(U),{ \ }\| \psi \|_{W^{1,p} (\Gamma_0)} < \infty \},\\
W^{2,p} ( \Gamma_0) & := \{ \psi : \Gamma_0 \to \mathbb{R} ;{ \ }\psi = \breve{\varphi}, { \ }\varphi \in W^{2,p}(U),{ \ }\| \psi \|_{W^{2,p} (\Gamma_0)} < \infty \}.
\end{align*}
Here
\begin{align}
\| \psi \|_{W^{1,p}( \Gamma_0)} & := \left( \int_{\Gamma_0} ( | \breve{\varphi}(x) |^p +  | \nabla_\Gamma \breve{\varphi}(x)  |^p ) { \ }d \mathcal{H}^2_x \right)^{\frac{1}{p}},\label{eq310}\\
\| \psi \|_{W^{2,p}( \Gamma_0)} & := \left( \int_{\Gamma_0} ( | \breve{\varphi}(x) |^p +  | \nabla_\Gamma \breve{\varphi}(x)|^p +  | \nabla_\Gamma^2 \breve{\varphi}(x)  |^p ) { \ }d \mathcal{H}^2_x \right)^{\frac{1}{p}}.\label{eq311}
\end{align}
Moreover, we define the inner product of $W^{1,2} (\Gamma_0)$ as follows: for $\psi_\sharp, \psi_\flat \in W^{1,2} (\Gamma_0)$
\begin{equation*}
\dual{\psi_\sharp , \psi_\flat}_{W^{1,2}} := \dual{\psi_\sharp , \psi_\flat} + \dual{\nabla_\Gamma \psi_\sharp , \nabla_\Gamma \psi_\flat}.
\end{equation*}

\begin{remark}\label{Rem35}
$(\mathrm{i})$ From \eqref{eq31}, \eqref{eq32}, \eqref{eq36}, \eqref{eq310}, we see that there is $C>0$ such that for each $\psi \in C^1 (\mathbb{R}^3)$ and $\varphi \in C^1 (\mathbb{R}^2)$ such that $\psi = \breve{\varphi}$
\begin{equation}\label{eq312}
\| \psi \|_{W^{1,p} (\Gamma_0 )} \leq C \| \varphi \|_{W^{1,p} (U)}.
\end{equation}
$(\mathrm{ii})$ From \eqref{eq31}-\eqref{eq33}, \eqref{eq36}, \eqref{eq310}, \eqref{eq311}, we see that there is $C>0$ such that for each $\psi \in C^2 (\mathbb{R}^3)$ and $\varphi \in C^2 (\mathbb{R}^2)$ such that $\psi = \breve{\varphi}$
\begin{equation}\label{eq313}
\| \psi \|_{W^{2,p} (\Gamma_0 )} \leq C \| \varphi \|_{W^{2,p} (U)}.
\end{equation}
$(\mathrm{iii})$ From Assumption \ref{ass24} and \eqref{eq35}, we see that there is $C>0$ such that for each $\psi \in C^1 (\mathbb{R}^3)$ and $\varphi \in C^1 (\mathbb{R}^2)$ such that $\psi = \breve{\varphi}$
\begin{align}
C^{-1} \| \nabla_X \varphi \|_{L^2(U)} \leq \| \nabla_\Gamma \psi \|_{L^2 (\Gamma_0)} \leq C \| \nabla_X \varphi \|_{L^2 (U)},\label{eq314}\\
C^{-1} \| \varphi \|_{W^{1,2}(U)} \leq \| \psi \|_{W^{1,2} (\Gamma_0)} \leq C \| \varphi \|_{W^{1,2} (U)}.\label{eq315}
\end{align}
\end{remark}

Using Remark \ref{Rem35}, we have
\begin{lemma}\label{Lem36}Let $1 \leq p < \infty$. Then\\
$(\mathrm{i})$ The function space $C_0(\Gamma_0)$ is dense in $L^p(\Gamma_0)$.\\
$(\mathrm{ii})$ The function space $C_0^1(\Gamma_0)$ is dense in $W_0^{1,p}(\Gamma_0)$.\\
$(\mathrm{iii})$ The function space $C^1 (\overline{ \Gamma_0 })$ is dense in $W^{1,p}(\Gamma_0)$.\\
$(\mathrm{iv})$ The function space $C^2 (\overline{ \Gamma_0 })$ is dense in $W^{2,p}(\Gamma_0)$.
\end{lemma}

\begin{proof}[Proof of Lemma \ref{Lem36}]
We only prove $(\mathrm{iii})$. Let $\psi \in W^{1,p} (\Gamma_0)$. By definition, there is $\varphi \in W^{1,p}(U)$ such that $\psi =\breve{ \varphi}$. Since $C^2(\overline{U})$ is dense in $W^{1,p} (U)$, there are $\varphi_k \in C^2 (\overline{U})$ such that
\begin{equation}\label{eq316}
\lim_{k \to \infty} \| \varphi - \varphi_k \|_{W^{1,p} (U)} = 0.
\end{equation}
Set $\psi_k = \breve{\varphi}_k$. Then $\psi_k \in C^2 (\overline{\Gamma_0})$. By \eqref{eq312} and \eqref{eq316}, we check that 
\begin{equation*}
\lim_{k \to \infty} \| \psi - \psi_k \|_{W^{1,p} (\Gamma_0)} = 0.
\end{equation*}
Therefore, the assertion $(\mathrm{iii})$ is proved.
\end{proof}

Let us study the trace operator $\gamma_p : W^{1,p} (\Gamma_0) \to L^p ( \partial \Gamma_0 )$.
\begin{definition}[Trace operator]\label{Def37}{ \ }\\
Fix $1 \leq p < \infty$. Define $\gamma_p : W^{1,p} (\Gamma_0) \to L^p ( \partial \Gamma_0 )$ as follows:\\
Let $\psi \in W^{1,p} ( \Gamma_0)$ and $\varphi \in W^{1,p} (U)$ such that $\psi = \breve{\varphi}$. By the assertion $(\rm{iii})$ of Lemma \ref{Lem36}, there are $\psi_k \in C^1 ( \overline{\Gamma_0})$ and $\varphi_k \in C^1 (\overline{U})$ such that $\psi_k = \breve{\varphi}_k$,
\begin{align*}
& \lim_{k \to \infty} \| \varphi - \varphi_k \|_{W^{1,p} (U)} = 0,\\
& \lim_{k \to \infty} \| \psi - \psi_k \|_{W^{1,p} (\Gamma_0)} = 0.
\end{align*}
From the definition of integral line, we see that
\begin{align*}
\int_{\partial \Gamma_0} |\psi_m|_{\partial \Gamma_0} - \psi_k|_{\partial \Gamma_0} |^p { \ }d \mathcal{H}_x^1 & = \int_{\partial U} |\varphi_m|_{\partial U} - \varphi_k|_{\partial U} |^p |n_1^U \mathfrak{g}_2 - n_2^U \mathfrak{g}_1 | { \ }d \mathcal{H}_X^1\\
& \leq C \| \varphi_m - \varphi_k \|_{W^{1,p} (U)}^p \to 0 \text{ (as }m, k \to \infty ). 
\end{align*}
Therefore, we define
\begin{equation*}
\gamma_p \psi = \lim_{k \to \infty} \breve{\varphi}_k \text{ in }L^p (\partial \Gamma_0).
\end{equation*}
Note that we observe that
\begin{align*}
\| \gamma_p \psi \|_{L^p ( \Gamma_0 )}^p &= \lim_{m \to \infty} \int_{\partial \Gamma_0} | { \ }\breve{\varphi}_m|_{\partial \Gamma_0}{ \ }|^p { \ }d \mathcal{H}^1_x\\
& = \lim_{m \to \infty} \int_{\partial U} | { \ } \varphi_m|_{\partial U} { \ }|^p |n_1^U \mathfrak{g}_2 - n_2^U \mathfrak{g}_1 | { \ } d \mathcal{H}^1_X\\
& = \int_{\partial U} | \widehat{\gamma}_p \varphi |^p |n_1^U \mathfrak{g}_2 - n_2^U \mathfrak{g}_1 | { \ } d \mathcal{H}^1_X \leq C \| \widehat{\gamma}_p \varphi \|_{L^p( \partial U )}^p.
\end{align*}
Here $\widehat{\gamma}_p : W^{1,p}(U) \to L^p (\partial U)$ is the trace operator. Note also that
\begin{equation*}
\gamma_p \psi = \psi|_{\partial \Gamma_0} \text{ if } \psi \in W^{1,p} (\Gamma_0) \cap C^1 ( \overline{\Gamma_0}).
\end{equation*}
\end{definition}

Using Lemmas \ref{Lem31}, \ref{Lem32}, Remark \ref{Rem35}, and Lemma \ref{Lem36}, we have the following two lemmas.
\begin{lemma}[Properties of $W^{1,p}(\Gamma_0)$ and $W^{2,p} ( \Gamma_0)$]\label{Lem38}{ \ }\\
$(\mathrm{i})$ Let $\psi \in W^{1,p} (\Gamma_0)$ and $\varphi \in W^{1,p}(U)$ such that $\psi = \breve{\varphi}$. Then for each $j=1,2,3$, $\partial_j^\Gamma \psi \in L^p (\Gamma_0)$ and
\begin{equation*}
\int_{\Gamma_0} \partial_j^\Gamma \psi { \ }d \mathcal{H}^2_x = \int_U \mathfrak{g}^{\alpha \beta} \frac{\partial \widehat{x}_j}{\partial X_\alpha} \frac{\partial \varphi}{\partial X_\beta} \sqrt{ \mathcal{G} } { \ }d X.
\end{equation*}
$(\mathrm{ii})$ For each $\psi_\sharp , \psi_\flat \in W^{1,p} (\Gamma_0)$ and $\varphi_\sharp, \varphi_\flat \in W^{1,p}(U)$ such that $\psi_\sharp = \breve{\varphi}_\sharp$ and $\psi_\flat = \breve{\varphi}_\flat$,
\begin{equation*}
\int_{\Gamma_0} \nabla_\Gamma \psi_\sharp \cdot \nabla_\Gamma \psi_\flat (x) { \ }d \mathcal{H}^2_x = \int_U \mathfrak{g}^{\alpha \beta} \frac{\partial \varphi_\sharp}{\partial X_\alpha} \frac{\partial \varphi_\flat}{\partial X_\beta} \sqrt{\mathcal{G}} { \ }d X.
\end{equation*}
$(\mathrm{iii})$ Let $\psi \in W^{2,p} (\Gamma_0)$ and $\varphi \in W^{2,p}(U)$ such that $\psi = \breve{\varphi}$. Then for each $i,j = 1,2,3$, $\partial_j^\Gamma \partial_i^\Gamma \psi \in L^p (\Gamma (t))$ and
\begin{equation*}
\int_{\Gamma_0} \partial_i^\Gamma \partial_j^\Gamma \psi { \ }d \mathcal{H}^2_x =  \int_U \mathfrak{g}^{\alpha' \beta'} \frac{\partial \widehat{x}_i}{\partial X_{\alpha'} } \frac{\partial'}{\partial X_{\beta'}} \left( \mathfrak{g}^{\alpha \beta} \frac{\partial \widehat{x}_j}{\partial X_\alpha} \frac{\partial \varphi }{\partial X_\beta} \right) \sqrt{ \mathcal{G} } { \ }d X.
\end{equation*}
$(\mathrm{iv})$ For each $\psi \in W^{2,p} (\Gamma_0)$ and $\varphi \in W^{2,p}(U)$ such that $\psi = \breve{\varphi}$,
\begin{equation*}
\int_{\Gamma_0} \Delta_\Gamma \psi (x) { \ }d \mathcal{H}^2_x = \int_U \left\{ \frac{1}{\sqrt{\mathcal{G}}} \frac{\partial}{\partial X_\alpha} \left( \sqrt{\mathcal{G}} \mathfrak{g}^{\alpha \beta} \frac{\partial \varphi }{\partial X_\beta} \right) \right\} \sqrt{\mathcal{G}} { \ }d X.
\end{equation*}
\end{lemma}

\begin{lemma}\label{Lem39}{ \ }\\
$(\mathrm{i})$ There is $C>0$ such that for each $\psi \in W^{1,p} (\Gamma_0)$ and $\varphi \in W^{1,p}(U)$ satisfying $\psi =\breve{ \varphi}$,
\begin{equation}\label{eq317}
\| \psi \|_{W^{1,p} (\Gamma_0 )} \leq C \| \varphi \|_{W^{1,p} (U)}.
\end{equation}
$(\mathrm{ii})$ There is $C>0$ such that for each $\psi \in W^{2,p} (\Gamma_0)$ and $\varphi \in W^{2,p}(U)$ satisfying $\psi =\breve{ \varphi}$,
\begin{equation}\label{eq318}
\| \psi \|_{W^{2,p} (\Gamma_0 )} \leq C \| \varphi \|_{W^{2,p} (U)}.
\end{equation}
$(\mathrm{iii})$ There is $C>0$ such that for each $\psi \in W^{1,2} (\Gamma_0)$ and $\varphi \in W^{1,2}(U)$ satisfying $\psi =\breve{ \varphi}$,
\begin{align}
& C^{-1} \| \nabla_X \varphi \|_{L^2(U)} \leq \| \nabla_\Gamma \psi \|_{L^2 (\Gamma_0)} \leq C \| \nabla_X \varphi \|_{L^2 (U)},\label{eq319}\\
& C^{-1} \| \varphi \|_{W^{1,2}(U)} \leq \| \psi \|_{W^{1,2} (\Gamma_0)} \leq C \| \varphi \|_{W^{1,2} (U)}.\label{eq320}
\end{align}
\end{lemma}
The proofs of Lemmas \ref{Lem38} and \ref{Lem39} are left for the readers. Next we study the formula for the integration by parts on the surface $\Gamma_0$.

\begin{lemma}[Formula for the integration by parts]\label{Lem310}
Let $1 < p,p' < \infty$ such that $1/p + 1/{p'}=1$. Then for each $j=1,2,3$, $f \in W^{1,p} ( \Gamma_0)$, and $\psi \in W^{1,p'} ( \Gamma_0)$,
\begin{equation}\label{eq321}
\int_{\Gamma_0} f ( \partial_j^\Gamma \psi ) { \ }d \mathcal{H}^2_x = - \int_{\Gamma_0} (\partial_j^\Gamma f + H_\Gamma n_j f ) \psi { \ }d \mathcal{H}^2_x + \int_{\partial \Gamma_0} \nu_j (\gamma_p f) (\gamma_{p'} \psi) { \ }d \mathcal{H}^1_x.
\end{equation}
Here $\gamma_p: W^{1,p} (\Gamma_0) \to L^p (\partial \Gamma_0)$ is the trace operator defined by Definition \ref{Def37}.
\end{lemma}

\begin{proof}[Proof of Lemma \ref{Lem310}]
We only derive \eqref{eq321}. Fix $j=1,2,3$. Let $1 < p,p' < \infty$ such that $1/p + 1/{p'}=1$. Fix $f \in W^{1,p} ( \Gamma_0)$ and $\psi \in W^{1,p'} ( \Gamma_0)$. By Lemma \ref{Lem36}, we find that there are $\phi, \phi_m \in W^{1,p}(U)$, $\varphi, \varphi_k \in W^{1,p'}(U)$ such that $f = \breve{\phi}$, $\psi = \breve{\varphi}$,
\begin{align}
& \lim_{m \to \infty} \| \phi - \phi_m \|_{W^{1,p} (U)} = 0,\label{eq322}\\
& \lim_{k \to \infty} \| \varphi - \varphi_k \|_{W^{1,p'} (U)} = 0.\label{eq323} 
\end{align}
Applying Lemma \ref{Lem32} with \eqref{eq317}, \eqref{eq322}, and \eqref{eq323}, we have
\begin{multline*}
\int_{\Gamma_0} f ( \partial_j^\Gamma \psi ) { \ }d \mathcal{H}^2_x = - \int_{\Gamma_0} (\partial_j^\Gamma f + H_\Gamma n_j f ) \psi { \ }d \mathcal{H}^2_x\\  + \lim_{m \to \infty} \lim_{k \to \infty} \int_{\partial \Gamma_0} \nu_j (\breve{\phi}_m|_{\partial \Gamma_0}) (\breve{\varphi}_k|_{\partial \Gamma_0}) { \ }d \mathcal{H}^1_x.
\end{multline*}
Now we prove that
\begin{equation}\label{eq324}
\lim_{m \to \infty} \lim_{k \to \infty} \int_{\partial \Gamma_0} \nu_j (\breve{\phi}_m|_{\partial \Gamma_0}) (\breve{\varphi}_k|_{\partial \Gamma_0}) { \ }d \mathcal{H}^1_x = \int_{\partial \Gamma_0} \nu_j (\gamma_p f) (\gamma_{p'} \psi) { \ }d \mathcal{H}^1_x.
\end{equation}
To this end, we first show that 
\begin{equation}\label{eq325}
\lim_{m \to \infty} \lim_{k \to \infty} \int_U (\phi_m|_{\partial U}) (\varphi_k|_{\partial U}) { \ }d \mathcal{H}^1_X = \int_{\partial U} (\widehat{\gamma}_p \phi ) ( \widehat{\gamma}_{p'} \varphi ) { \ } d \mathcal{H}^1_X.
\end{equation}
Here $\widehat{\gamma}_p:W^{1,p} (U) \to L^p(\partial U)$ is the trace operator. Using the H$\ddot{\rm{o}}$lder inequality, we check that
\begin{multline*}
\left| \int_{\partial U} \{ (\widehat{\gamma}_p \phi ) ( \widehat{\gamma}_{p'} \varphi ) - (\phi_m|_{\partial U}) (\varphi_k|_{\partial U}) \} { \ } d \mathcal{H}^1_X\right|\\
 = \left| \int_{\partial U} (\widehat{\gamma}_p \phi - \phi_m|_{\partial U} ) ( \widehat{\gamma}_{p'} \varphi ) { \ } d \mathcal{H}^1_X + \int_{\partial U} \phi_m|_{\partial U} ( \widehat{\gamma}_{p'} \varphi - \varphi_k|_{\partial U} ) { \ } d \mathcal{H}^1_X \right|\\
\leq \| \widehat{\gamma}_p \phi - \phi_m|_{\partial U} \|_{L^p (\partial U)} \| \widehat{\gamma}_{p'} \varphi \|_{L^{p'}( \partial U)} + \| \phi_m|_{\partial U} \|_{L^p (\partial U)} \| \varphi_k|_{\partial U} - \widehat{\gamma}_{p'} \varphi \|_{L^{p'}( \partial U)}\\
\leq \| \phi - \phi_m \|_{W^{1,p}(U)} \| \varphi \|_{W^{1,p'} (U)} +  \| \phi_m \|_{W^{1,p}(U)} \| \varphi_k - \varphi \|_{W^{1,p'} (U)}.
\end{multline*} 
Using \eqref{eq322} and \eqref{eq323}, we have \eqref{eq325}. Since $\nu \cdot \nu =1$ and
\begin{equation*}
\int_{\partial \Gamma_0} \nu_j (\breve{\phi}_m|_{\partial \Gamma_0}) (\breve{\varphi}_k|_{\partial \Gamma_0}) { \ }d \mathcal{H}^1_x = \int_{\partial U} \widehat{\nu}_j (\phi_m|_{\partial U}) (\varphi_k|_{\partial U}) | n_1^U \mathfrak{g}_2 - n_2^U \mathfrak{g}_1 | { \ }d \mathcal{H}^1_X,
\end{equation*}
we apply \eqref{eq325} to have \eqref{eq324}. Therefore, the lemma follows.
\end{proof}

Let us derive the key lemma to show the existence of strong solutions to system \eqref{eq11}.
\begin{lemma}\label{Lem311}
Let $1 < p < \infty$, and let $v \in W^{1,p} (\Gamma_0)$. Then there exists $u \in W^{2,p} (\Gamma_0)$ such that
\begin{equation*}
\gamma_p v = \gamma_p u.
\end{equation*}
Here $\gamma_p: W^{1,p} (\Gamma_0) \to L^p (\partial \Gamma_0)$ is the trace operator defined by Definition \ref{Def37}.
\end{lemma}
\begin{proof}[Proof of Lemma \ref{Lem311}]
Fix $v \in W^{1,p} (\Gamma_0)$. From the assertion $(\rm{iii})$ of Lemma \ref{Lem36}, there are $v_m \in C^1 (\overline{\Gamma_0})$, $\varphi \in W^{1,p}(U)$, and $\varphi_m \in C^1(\overline{U})$ such that $v = \breve{\varphi}$, $v_m = \breve{\varphi}_m$,
\begin{align*}
\lim_{m \to \infty} \| v - v_m \|_{W^{1,p} (\Gamma_0)} =0,\\
\lim_{m \to \infty} \| \varphi - \varphi_m \|_{W^{1,p} (U)} =0.
\end{align*}
Let $\gamma_p$ and $\widehat{\gamma}_p$ be the two trace operators such as $\gamma_p : W^{1,p} (\Gamma_0) \to L^p (\partial \Gamma_0)$ and $\widehat{\gamma}_p : W^{1,p} (U) \to L^p (\partial U)$. Now we consider
\begin{equation}\label{eq326}
\begin{cases}
- \Delta_X \mathcal{J} = 0 \text{ in } U,\\
\mathcal{J} |_{\partial U} = \widehat{\gamma}_p \varphi.
\end{cases}
\end{equation}
Applying the Green function for the Laplacian $- \Delta_X$ in the domain $U$, we see that system \eqref{eq326} admits a unique strong solution $\mathcal{J} \in W^{2,p} (U)$. Since $\mathcal{J} \in W^{2,p} (U)$, there are $\mathcal{J}_m \in C^2 ( \overline{U})$ such that
\begin{equation*}
\lim_{m \to \infty} \| \mathcal{J} - \mathcal{J}_m \|_{W^{2,p} (U)} = 0.
\end{equation*}
Set $u = \breve{\mathcal{J}}$ and $u_m = \breve{\mathcal{J}}_m$. Then we see that $u \in W^{2,p}( \Gamma_0 )$ and that
\begin{equation*}
\lim_{m \to \infty} \| u - u_m \|_{W^{2,p} ( \Gamma_0 )} = 0.
\end{equation*}

Next we show that $\gamma_p u = \gamma_p v$. Since
\begin{align*}
\| \widehat{\gamma}_p \varphi_m - \widehat{\gamma}_p \mathcal{J}_m \|_{L^p (\partial U)} & \leq C \| \widehat{\gamma}_p \varphi_m - \widehat{\gamma}_p \varphi \|_{L^p (\partial U)} + C \| \widehat{\gamma}_p \varphi - \widehat{\gamma}_p \mathcal{J}_m \|_{L^p (\partial U)}\\
& \to 0 \text{ as } m \to \infty,
\end{align*}
we check that
\begin{align*}
\| \gamma_p v - \gamma_p u \|_{L^p (\partial \Gamma_0)}^p & = \lim_{m \to \infty} \int_{\partial U} |\varphi_m|_{\partial U} - \mathcal{J}_m |_{\partial U} |^p |n_1^U \mathfrak{g}_2 - n_2^U \mathfrak{g}_1 | { \ }d \mathcal{H}_X^1\\
& \leq C \lim_{m \to \infty} \int_{\partial U} |\varphi_m|_{\partial U} - \mathcal{J}_m |_{\partial U} |^p { \ }d \mathcal{H}_X^1\\
& \leq C \lim_{m \to \infty} \| \widehat{\gamma}_p \varphi_m - \widehat{\gamma}_p \mathcal{J}_m \|_{L^p (\partial U)} = 0.
\end{align*}
Therefore, we conclude that $\gamma_p u = \gamma_p v$, and the lemma follows.
\end{proof}

Applying Remark \ref{Rem33} and Lemma \ref{Lem39}, we have
\begin{lemma}\label{Lem312}{ \ }\\
$(\mathrm{i})$ Let $1 \leq p < \infty$. Then $L^p(\Gamma_0)$ is a Banach space.\\
$(\mathrm{ii})$ Each function space $L^2(\Gamma_0)$, $W_0^{1,2}(\Gamma_0)$, and $W^{1,2}(\Gamma_0)$ is a Hilbert space.
\end{lemma}
\noindent The proof of Lemma \ref{Lem312} is left for the readers.

Finally, we state the surface Rellich-Kondrakov theorem.
\begin{lemma}[Sufrace Rellich-Kondrakov theorem]\label{Lem313}
Assume that $\{ \psi_k \}_{k \in \mathbb{N}}$ is bounded in $W^{1,2} (\Gamma_0)$. Then there are $\psi \in L^2 (\Gamma_0)$ and a subsequence $\{ \psi_m \}_{m \in \mathbb{N}} \subset \{ \psi_k \}_{k \in \mathbb{N}}$ such that
\begin{equation}\label{eq327}
\lim_{m \to \infty } \| \psi_m - \psi \|_{L^2( \Gamma_0)} = 0.
\end{equation}
\end{lemma}

\begin{proof}[Proof of Lemma \ref{Lem313}]
Assume that $\{ \psi_k \}_{k \in \mathbb{N}}$ is bounded in $W^{1,2} (\Gamma_0)$. By definition and \eqref{eq320}, there is a sequence $\{ \varphi_k \}_{k \in \mathbb{N}} \subset W^{1,2}(U)$ such that $\psi_k = \breve{\varphi}_k$. Since
\begin{equation*}
\| \varphi_k \|_{W^{1,2} (U)} \leq C \| \psi_k \|_{W^{1,2} (\Gamma_0)} < +\infty,
\end{equation*}
we see that $\{ \varphi_k \}_{k \in \mathbb{N}}$ is bounded in $W^{1,2} (U)$. Since $\partial U$ is $C^2$-class, it follows from the Rellich-Kondrakov theorem to see that there are $\varphi \in W^{1,2} (U)$ and a subsequence $\{ \varphi_m \}_{m \in \mathbb{N}} \subset \{ \varphi_k \}_{k \in \mathbb{N}}$ such that
\begin{equation*}
\lim_{m \to \infty } \| \varphi_m - \varphi \|_{L^2(U)} = 0.
\end{equation*}
Set $\psi = \breve{\varphi}$ and $\psi_m = \breve{\varphi}_m$. It is clear that $\psi \in L^2 (\Gamma_0)$. We also check that $\{ \psi_m \}_{m \in \mathbb{N}}$ is a subsequence of $\{ \psi_k \}_{k \in \mathbb{N}}$ and that \eqref{eq327} holds. Therefore, the lemma follows.
\end{proof}

\subsection{Function spaces $L_0^p (\Gamma_0)$ and $H^1 (\Gamma_0)$}\label{subsec33}

We introduce and study new function spaces. For each $1 \leq p < \infty$,
\begin{equation*}
L_0^p (\Gamma_0) = \left\{ f \in L^p (\Gamma_0) ; { \ } \int_{\Gamma_0} f (y) { \ }d \mathcal{H}^2_y = 0 \right\}
\end{equation*}
with the norm $\| \cdot \|_{L_0^p (\Gamma_0)} := \| \cdot \|_{L^p (\Gamma_0)}$. Set
\begin{equation*}
H^1 (\Gamma_0) = L_0^2 (\Gamma_0 ) \cap W^{1,2} (\Gamma_0 ) = \left\{ f \in W^{1,2} (\Gamma_0) ; { \ } \int_{\Gamma_0} f (y) { \ }d \mathcal{H}^2_y = 0 \right\}
\end{equation*}
with the norm $\| \cdot \|_{H^1 (\Gamma_0)} := \| \cdot \|_{W^{1,2} (\Gamma_0)}$ and the inner product $\dual{\cdot , \cdot }_{H^1 (\Gamma_0)} := \dual{\cdot, \cdot} + \dual{\nabla_\Gamma \cdot, \nabla_\Gamma \cdot}$. 
\begin{lemma}\label{Lem314} Let $1 \leq p < \infty$. Then\\
$(\mathrm{i})$ The function space $L_0^p ( \Gamma_0 )$ is a Banach space.\\
$(\mathrm{ii})$ The function space $L_0^p ( \Gamma_0 ) \cap W^{1,p}(\Gamma_0 )$ is dense in $L_0^p ( \Gamma_0 )$.\\
$(\mathrm{iii})$ The function space $L^2_0 (\Gamma_0)$ is a Hilbert space.\\
$(\mathrm{iv})$ The function space $H^1 ( \Gamma_0)$ is a Hilbert space. Moreover, $H^1 (\Gamma_0)$ is dense in $L^2_0 (\Gamma_0 )$.
\end{lemma}

\begin{proof}[Proof of Lemma \ref{Lem314}]
We first show $(\mathrm{i})$. Let $\mu_1 , \mu_2 \in \mathbb{R}$ and $f_1 , f_2 \in L_0^p (\Gamma_0)$. By the definition of $L_0^p (\Gamma_0)$, we find that
\begin{equation*}
\int_{\Gamma_0} \{ \mu_1 f_1 (y) + \mu_2 f_2(y) \} { \ } d \mathcal{H}^2_y = \mu_1 \int_{\Gamma_0} f_1 (y) { \ }d \mathcal{H}^2_y + \mu_2 \int_{\Gamma_0} f_2(y) { \ } d \mathcal{H}^2_y = 0.
\end{equation*}
Therefore we see that $\mu_1 f_1 + \mu_2 f_2 \in L_0^p ( \Gamma_0 )$. Let $\{ f_k \}_{k \in \mathbb{N}} \subset L_0^p ( \Gamma_0 )$ be a Cauchy sequence, that is,
\begin{equation*}
\lim_{k,k_* \to \infty} \| f_k - f_{k_* } \|_{L^p (\Gamma_0 ) } = 0.
\end{equation*}
Since $L^p ( \Gamma_0) $ is a Banach space, there exists $f_* \in L^p ( \Gamma_0)$ such that
\begin{equation}\label{eq328}
\lim_{k \to \infty} \| f_k - f_* \|_{L^p (\Gamma_0 ) } = 0.
\end{equation}
Now we prove that $f_* \in L_0^p ( \Gamma_0 )$. Using the surface H$\ddot{\rm{o}}$lder inequality \eqref{eq39} and \eqref{eq328}, we see that
\begin{align*}
\left| \int_{\Gamma_0} f_* (y) { \ } d \mathcal{H}^2_y \right| = & \left| \int_{\Gamma_0} \{ f_* (y) - f_k (y) \} { \ } d \mathcal{H}^2_y \right|\\
\leq & | \Gamma_0 |^{1 - \frac{1}{p}} \| f_* -f \|_{L^p ( \Gamma_0 )} \to 0 \text{ (as } k \to \infty) .
\end{align*}
This implies that
\begin{equation*}
\int_{\Gamma_0} f_* (y) { \ } d \mathcal{H}^2_y = 0.
\end{equation*}
Therefore, we see that $f_* \in L_0^p ( \Gamma_0 )$. Therefore, we conclude that $L_0^p ( \Gamma_0 )$ is a Banach space.

Next, we prove $(\mathrm{ii})$. Let $f \in L_0^p (\Gamma_0 )$. Since $C_0^1 ( \Gamma_0 )$ is dense in $L^p (\Gamma_0)$, there are $f_m \in C_0^1 (\Gamma_0 )$ such that
\begin{equation}\label{eq329}
\lim_{m \to \infty }\| f - f_m \|_{L^p ( \Gamma_0 )} = 0.
\end{equation}
Set
\begin{equation*}
\phi_m = f_m - \frac{1}{| \Gamma_0 | }\int_{\Gamma_0} f_m (y) { \ }d \mathcal{H}^2_y .
\end{equation*}
It is easy to check that $\phi_m \in W^{1,p} (\Gamma_0)$ and
\begin{equation*}
\int_{\Gamma_0} f (y) { \ } d \mathcal{H}^2_y = 0.
\end{equation*}
Since $\int_{\Gamma_0} f (y) { \ } d \mathcal{H}^2_y = 0$, we apply the surface H$\ddot{\rm{o}}$lder inequality and \eqref{eq329} to see that
\begin{align*}
\left| \int_{\Gamma_0} f_m (y) { \ }d \mathcal{H}^2_y \right| & = \left| \int_{\Gamma_0} f_m (y) - f (y) { \ }d \mathcal{H}^2_y \right|\\
& \leq | \Gamma_0|^{1 - \frac{1}{p}} \| f - f_m \|_{L^p (\Gamma_0)}\\
& \to 0 \text{ (as }m \to \infty).
\end{align*}
Thus, we have
\begin{equation*}
\lim_{m \to \infty } \left| \int_{\Gamma_0} f_m (y) { \ }d \mathcal{H}^2_y \right| = 0.
\end{equation*}
Using $(a + b)^p \leq 2^p a^p + 2^p b^p$, we check that
\begin{align*}
\| f - \phi_m \|_{L^p (\Gamma_0)}^p = \int_{\Gamma_0}\left| f(x) - f_m(x) + \frac{1}{|\Gamma_0 |} \int_{\Gamma_0} f_m (y) { \ }d \mathcal{H}^2_y \right|^p { \ }d \mathcal{H}^2_x\\
\leq 2^p \| f - f_m \|_{L^p (\Gamma_0)} + 2^p \int_{\Gamma_0} \left| \frac{1}{|\Gamma_0|} \int_{\Gamma_0} f_m (y) { \ }d \mathcal{H}^2_y \right|^p { \ }d \mathcal{H}^2_x\\
\to 0 \text{ (as } m \to \infty) .
\end{align*}
Therefore, the assertion $(\mathrm{ii})$ is proved. The proofs of the assertions $(\mathrm{iii})$ and $(\mathrm{iv})$ are left for the readers.
\end{proof}

Next, we introduce the surface Poincar\'{e} inequalities.
\begin{lemma}[Surface Poincar\'{e} inequalities]\label{Lem315}{ \ }\\
$(\mathrm{i})$ There is $C_\star = C_\star (\Gamma_0) >0$ such that for all $f \in H^1 (\Gamma_0)$
\begin{equation}\label{eq330}
\| f \|_{L^2 ( \Gamma_0 )} \leq C_\star \| \nabla_\Gamma f \|_{L^2 (\Gamma_0)} .
\end{equation}
$(\mathrm{ii})$ There is $C_\bigstar = C_\bigstar (\Gamma_0) >0$ such that for all $\mathfrak{f} \in W_0^{1,2} (\Gamma_0)$
\begin{equation}\label{eq331}
\| \mathfrak{f} \|_{L^2 ( \Gamma_0 )} \leq C_\bigstar \| \nabla_\Gamma \mathfrak{f} \|_{L^2 (\Gamma_0)} .
\end{equation}
\end{lemma}
\noindent See \cite{Mc70}, \cite[Chapter 2, Corollary 4.3]{Aub82}, \cite[\S 5]{Str83}, \cite[\S 3.3]{Heb96}, and \cite{RT15} for another type of surface Poincar\'e inequalities. By developing some ideas in \cite[\S 5.8]{Eva10}, we prove Lemma \ref{Lem315}.
\begin{proof}[Proof of Lemma \ref{Lem315}]
We first show $(\mathrm{i})$. To derive \eqref{eq330}, we show that there is $C_\Gamma >0$ such that for all $f \in W^{1,2} (\Gamma_0)$
\begin{equation}\label{eq332}
\left\| f - \frac{1}{|\Gamma_0|} \int_{\Gamma_0} f (y) { \ }d \mathcal{H}^2_y \right\|_{L^2 ( \Gamma_0 )} \leq C_\Gamma \| \nabla_\Gamma f \|_{L^2 (\Gamma_0)} .
\end{equation}
We argue by contradiction. Assume that \eqref{eq332} is not true. Then there are $\{ f_k \}_{k \in \mathbb{N}} \subset W^{1,2} ( \Gamma_0)$ such that
\begin{equation}\label{eq333}
\left\| f_k - \frac{1}{|\Gamma_0|} \int_{\Gamma_0} f_k (y) { \ }d \mathcal{H}^2_y \right\|_{L^2 ( \Gamma_0 )} > k \| \nabla_\Gamma f_k \|_{L^2 (\Gamma_0)} .
\end{equation}
Set
\begin{equation*}
\psi_k = \frac{f_k - \frac{1}{|\Gamma_0|} \int_{\Gamma_0} f_k (y) { \ }d \mathcal{H}^2_y }{\left\| f_k - \frac{1}{|\Gamma_0|} \int_{\Gamma_0} f_k (y) { \ }d \mathcal{H}^2_y \right\|_{L^2 ( \Gamma_0 )}}.
\end{equation*}
It is clear that
\begin{align}
\| \psi_k \|_{L^2 (\Gamma_0)} = 1,\label{eq334}\\
\int_{\Gamma_0} \psi_k (y) { \ }d \mathcal{H}_y^2 = 0.\label{eq335} 
\end{align}
By \eqref{eq333} and \eqref{eq334}, we see that
\begin{equation}\label{eq336}
\| \nabla_k \psi_k \|_{L^2 (\Gamma_0)} < \frac{1}{k}.
\end{equation}
From \eqref{eq334} and \eqref{eq336}, we find that $\{ \psi_k \}_{k \in \mathbb{N}}$ is bounded in $W^{1,2} (\Gamma_0)$. From the Rellich-Kondrakov theorem, there are $\psi \in L^2 (\Gamma_0)$ and $\{ \psi_m \}_{m \in \mathbb{N}} \subset \{ \psi_k \}_{k \in \mathbb{N}} $ such that
\begin{equation}\label{eq337}
\lim_{m \to \infty} \| \psi_m - \psi \|_{L^2 ( \Gamma_0 )} = 0 .
\end{equation}
Using \eqref{eq337}, \eqref{eq334}, and \eqref{eq335}, we check that
\begin{align}
\| \psi \|_{L^2 ( \Gamma_0 )} = 1,\label{eq338}\\
\int_{\Gamma_0} \psi (y) { \ }d \mathcal{H}^2_y = 0.\label{eq339}
\end{align}
Now we show that $\psi \in W^{1,2} (\Gamma_0)$. Fix $\phi \in C_0^1 (\Gamma_0)$. Using \eqref{eq336}, we see that for each $j = 1,2,3$,
\begin{align*}
- \int_{\Gamma_0} \psi(\partial_j^\Gamma \phi + H_\Gamma n_j \phi ) { \ }d \mathcal{H}^2_x & = - \int_{\Gamma_0} \psi_m (\partial_j^\Gamma \phi + H_\Gamma n_j \phi ) { \ }d \mathcal{H}^2_x\\
& = \int_{\Gamma_0} (\partial_j^\Gamma \psi_m ) \phi { \ }d \mathcal{H}^2_x\\
& \to 0 { \ }(\text{as } m \to \infty).
\end{align*}
Since $\phi$ is arbitrary, we see that $\psi \in W^{1,2} (\Gamma_0)$ and $\| \nabla_\Gamma \psi \|_{L^2 (\Gamma_0)} = 0$. Since
\begin{equation*}
0 = \| \nabla_\Gamma \psi \|_{L^2 (\Gamma_0)} \geq C \| \nabla_X \widehat{\psi} \|_{L^2 (U)},
\end{equation*}
we find that $\psi$ is a constant. From \eqref{eq339}, we see that $\psi =0$, which contradicts the condition \eqref{eq338}. Therefore, we see $(\mathrm{i})$.

Next we show $(\mathrm{ii})$. Let $\mathfrak{f} \in W_0^{1,2} (\Gamma_0)$. By definition, there is $\varphi \in W_0^{1,2} (U)$ such that $\mathfrak{f} = \breve{\varphi}$. Using the Poincar\'{e} inequality for $W_0^{1,2} (U)$-functions with \eqref{eq320}, we check that
\begin{align*}
\| \mathfrak{f} \|_{L^2 (\Gamma_0)} & \leq C \| \varphi \|_{L^2 (U)}\\
& \leq C(U) \| \nabla_X \varphi \|_{L^2 (U)}\\
& \leq C \| \nabla_\Gamma \mathfrak{f} \|_{L^2 (\Gamma_0)}.
\end{align*}
Therefore, we have \eqref{eq331}.
\end{proof}

\section{Existence of weak solutions}\label{sect4}
Let us show the existence of a unique weak solution to system \eqref{eq11} when $F \in L_0^2 (\Gamma_0)$. We apply the Lax-Milgram theorem and some properties of the function space $H^1 ( \Gamma_0 )$ to prove Proposition \ref{prop28}. We first prepare the following two lemmas.
\begin{lemma}\label{lem41}
Let $F \in L_0^2 (\Gamma_0)$. Then there exists a unique function $v \in H^1 ( \Gamma_0 )$ such that for all $\psi \in H^1 (\Gamma_0)$
\begin{equation}\label{eq41}
\dual{\nabla_\Gamma v , \nabla_\Gamma \psi } = \dual{F , \psi}.
\end{equation}
\end{lemma}

\begin{lemma}\label{lem42}
Let $F \in L_0^2 (\Gamma_0)$ and $v \in H^1 (\Gamma_0)$. Assume that \eqref{eq41} holds for all $\psi \in H^1 ( \Gamma_0 )$. Then \eqref{eq41} holds for all $\psi \in W^{1,2} ( \Gamma_0 )$.
\end{lemma}

Let us first attack Lemma \ref{lem41}.
\begin{proof}[Proof of Lemma \ref{lem41}]
We first consider the case when $F_* \in H^1 (\Gamma_0)$. For all $\psi \in H^1 (\Gamma_0)$, set
\begin{equation*}
F_* (\psi) = \dual{F_* , \psi }.
\end{equation*}
It is clear that $F_*: H^1 (\Gamma_0) \to \mathbb{R}$ is a bounded linear functional on $H^1 (\Gamma_0)$. For all $v , \psi \in H^1 ( \Gamma_0)$,
\begin{equation*}
\mathcal{B} ( v , \psi ) := \dual{\nabla_\Gamma v , \nabla_\Gamma \psi}.
\end{equation*}
It is easy to check that $\mathcal{B} (\cdot , \cdot )$ is a bilinear mapping on $H^1(\Gamma_0)$, and that for all $v , \psi \in H^1 (\Gamma_0)$
\begin{equation*}
| \mathcal{B} ( v , \psi ) | \leq \| v \|_{W^{1,2} (\Gamma_0)} \| \psi \|_{W^{1,2} (\Gamma_0)}.
\end{equation*}
Using the surface Poincar\'{e} inequality \eqref{eq330}, we observe that for all $v \in H^1(\Gamma_0 )$
\begin{equation*}
| \mathcal{B} ( v , v ) | \geq \frac{1}{C_\star^2 + 1} \| v \|_{W^{1,2} (\Gamma_0)}^2,
\end{equation*}
where $C_\star$ is the positive constant appearing in \eqref{eq330}.

Since $H^1 ( \Gamma_0)$ is a Hilbert space, it follows from the Lax-Milgram theorem to see that there exists a unique function $v \in H^1 (\Gamma_0)$ such that for all $\psi \in H^1 (\Gamma_0)$
\begin{equation*}
F_* ( \psi ) =\dual{F_* , \psi } = \dual{ \nabla_\Gamma v , \nabla_\Gamma \psi }. 
\end{equation*}

Let $F \in L_0^2 (\Gamma_0)$. From Lemma \ref{Lem314}, there is $\{ F_m \}_{m \in \mathbb{N}} \subset H^1 (\Gamma_0 )$ such that
\begin{equation}\label{eq42}
\lim_{m \to \infty} \| F - F_m \|_{L^2 (\Gamma_0 )} = 0.
\end{equation} 
By the previous argument, we find that for each $F_m$ there exists a unique function $v_m \in H^1 (\Gamma_0)$ such that for all $\psi \in H^1 (\Gamma_0 )$
\begin{equation*}
\dual{F_m , \psi } = \dual{\nabla_\Gamma v_m , \nabla_\Gamma \psi }.
\end{equation*}
Since $\partial_\Gamma^j$ is a closed operator and $v_m \in H^1 (\Gamma_0 )$, we use \eqref{eq42} to see that there exists a unique function $v \in H^1 (\Gamma_0)$ such that for all $\psi \in H^1 (\Gamma_0)$
\begin{equation*}
\dual{F , \psi } = \dual{\nabla_\Gamma v , \nabla_\Gamma \psi }.
\end{equation*}
Therefore, the lemma follows.
\end{proof}

Next, we prove Lemma \ref{lem42}.
\begin{proof}[Proof of Lemma \ref{lem42}]
Let $\psi \in W^{1,2} (\Gamma_0)$. Set
\begin{align*}
\psi_1 &:= \psi - \frac{1}{ | \Gamma_0 |} \int_{\Gamma_0} \psi (y) { \ }d \mathcal{H}^2_y,\\
\psi_2 & := \frac{1}{ | \Gamma_0 |} \int_{\Gamma_0} \psi (y) { \ }d \mathcal{H}^2_y.
\end{align*}
It is clear that $\psi_1 \in H^1 (\Gamma_0)$ and $\psi_2 \in \mathbb{R}$. By assumption, we check that
\begin{align*}
\dual{\nabla_\Gamma v , \nabla_\Gamma \psi } & = \dual{\nabla_\Gamma v , \nabla_\Gamma (\psi_1 + \psi_2) }\\
& = \dual{\nabla_\Gamma v , \nabla_\Gamma \psi_1 } = \dual{F , \psi_1}.
\end{align*}
Since $F \in L_0^2 ( \Gamma_0)$ and $\psi_2 \in \mathbb{R}$, we find that
\begin{align*}
\dual{F , \psi } & = \dual{F , \psi_1 + \psi_2 }\\
& = \dual{F , \psi_1}.
\end{align*}
As a result, we have
\begin{equation}\label{eq43}
\dual{\nabla_\Gamma v , \nabla_\Gamma \psi } = \dual{F , \psi }.  
\end{equation}
Since $\psi$ is arbitrary, we conclude that \eqref{eq43} holds for all $\psi \in W^{1,2} (\Gamma_0)$. Therefore, the lemma follows.
\end{proof}

Finally, we prove Proposition \ref{prop28}.
\begin{proof}[Proof of Proposition \ref{prop28}]
Let $F \in L_0^2 (\Gamma_0)$. From Lemmas \ref{lem41} and \ref{lem42}, we see the existence of a weak solution $v \in H^1 (\Gamma_0)$ to system \eqref{eq11}. Now we discuss the uniqueness of the weak solutions to \eqref{eq11}. Let $v_* \in H^1 (\Gamma_0)$. Assume that
\begin{equation*}
\dual{\nabla_\Gamma v_* , \nabla_\Gamma \psi } = \dual{F , \psi }
\end{equation*}
holds for all $\psi \in W^{1,2} (\Gamma_0)$. Set $v_\sharp = v - v_*$. Then
\begin{equation*}
\dual{\nabla_\Gamma v_\sharp , \nabla_\Gamma \psi } = 0
\end{equation*}
holds for all $\psi \in W^{1,2} (\Gamma_0)$. Now we take $\psi = v_\sharp$. Then
\begin{equation*}
\| \nabla_\Gamma v_\sharp \|_{L^2 (\Gamma_0)} = 0.
\end{equation*}
Since $v_\sharp \in H^1 (\Gamma_0)$, we use the surface Poincar\'{e} inequality \eqref{eq330} to see that
\begin{equation*}
\| v_\sharp \|_{L^2 (\Gamma_0)} = 0.
\end{equation*}
This implies that $v_\sharp = 0$, that is, $v_* = v$. Therefore, Proposition \ref{prop28} is proved.
\end{proof}

\section{Regularity for weak solutions to the surface Poisson equation}\label{sect5}
Let us study the regularity for weak solutions to the surface Poisson equation. Applying a standard regularity theory for elliptic equations, we investigate the regularity for our weak solutions. However, we need some techniques to derive \eqref{EQ515} since the surface $\Gamma_0$ has the boudary $\partial \Gamma_0$. This section provides a detailed proof of Poposition \ref{prop29}. Let us first prepare the two lemmas.
\begin{lemma}\label{lem51}
Fix $P = { }^t (P_1 , P_2) \in \partial U$. Then there are $\delta >0$, $\mathfrak{b} \in C^2 (\mathbb{R})$, and $\Phi \in C^2 (B_\delta (P) \cap U )$ such that
\begin{align*}
& \Phi : B_\delta (P) \cap U \to Q_+^{\delta , \mathfrak{b}},\\
& \Phi^{-1} : Q_+^{\delta , \mathfrak{b}} \to B_\delta (P),\\
& \Phi^{-1} \in C^2 (Q_+^{\delta , \mathfrak{b}}),
\end{align*}
where $\Phi^{-1}$ is the inverse function of $\Phi$,
\begin{align*}
B_\delta (P) & = \{ X = { }^t (X_1 , X_2) \in \mathbb{R}^2; { \ } (X_1 - P_1)^2 + (X_2 - P_2)^2 < \delta^2 \},\\ 
Q_+^{\delta, \mathfrak{b}} & = \{ Y = { }^t (Y_1 , Y_2 ) \in \mathbb{R}^2; { \ } -\delta < Y_1 < \delta ,{ \ } 0 < Y_2 < \delta - \mathfrak{b} (Y_1) \}.
\end{align*}
Moreover, $\Phi$ and $\Phi^{-1}$ satisfy
\begin{align}
& \rm{det} (\nabla_X \Phi (X)) = 1 \text{ for all } X \in B_\delta (P) \cap U,\label{EQ51}\\
& \rm{det} (\nabla_Y \Phi^{-1} (Y)) = 1\text{ for all } Y \in Q_+^{\delta , \mathfrak{b}} .\label{EQ52}
\end{align}
\end{lemma}

\begin{lemma}\label{lem52}
Let $\mathfrak{b} \in C^2 ( \mathbb{R})$, $P_1 , P_2 , \theta \in \mathbb{R}$. For each $Y = { }^t (Y_1,Y_2) \in \mathbb{R}^2$ and $\alpha, \beta \in \{1 , 2\}$, 
\begin{equation*}
\Psi = \Psi (Y) = \begin{pmatrix} \Psi_1\\ \Psi_2 \end{pmatrix} := \begin{pmatrix}
Y_1 \cos \theta + Y_2 \sin \theta + \mathfrak{b} (Y_1) \sin \theta + P_1 \\
- Y_1 \sin \theta + Y_2 \cos \theta + \mathfrak{b} (Y_1) \cos \theta + P_2
\end{pmatrix},
\end{equation*}
\begin{equation*}
\widetilde{x} = \widetilde{x} (Y)=
\begin{pmatrix}
\widetilde{x}_1\\
\widetilde{x}_2\\
\widetilde{x}_3
\end{pmatrix}
:= \widehat{x} (\Psi (Y)),
\end{equation*}

\begin{align*}
\mathring{g}_\alpha & := \mathring{g}_\alpha (Y) = \frac{\partial \Psi}{\partial Y_\alpha},{ \ }\mathring{g}_{\alpha \beta} := \mathring{g}_\alpha \cdot \mathring{g}_\beta,{ \ }\widetilde{g}_\alpha := \widetilde{g}_\alpha (Y) = \frac{\partial \widetilde{x}}{\partial Y_\alpha}, { \ }\widetilde{g}_{\alpha \beta} := \widetilde{g}_\alpha \cdot \widetilde{g}_\beta,\\
\widetilde{\mathfrak{g}}_{\alpha \beta} & =\widetilde{\mathfrak{g}}_{\alpha \beta}(Y) := \mathfrak{g}_{\alpha \beta} (\Psi (Y)) ,{ \ }\widetilde{\mathfrak{g}}^{\alpha \beta} =\widetilde{\mathfrak{g}}^{\alpha \beta}(Y) := \mathfrak{g}^{\alpha \beta} (\Psi (Y)).
\end{align*}
Set
\begin{align*}
\mathring{ \mathcal{G}} := \mathring{g}_{11} \mathring{g}_{22} - \mathring{g}_{12} \mathring{g}_{21},{ \ }( \mathring{g}^{\alpha \beta} )_{2 \times 2} := ( \mathring{g}_{ \alpha \beta } )_{2 \times 2}^{-1} = \frac{1}{ \mathring{ \mathcal{G} }}
\begin{pmatrix}
\mathring{g}_{22} & - \mathring{g}_{12}\\
- \mathring{g}_{12} & \mathring{g}_{11}
\end{pmatrix},\\
\widetilde{ \mathcal{G}} := \widetilde{g}_{11} \widetilde{g}_{22} - \widetilde{g}_{12} \widetilde{g}_{21},{ \ }( \widetilde{g}^{\alpha \beta} )_{2 \times 2} := ( \widetilde{g}_{ \alpha \beta } )_{2 \times 2}^{-1} = \frac{1}{ \widetilde{ \mathcal{G} }}
\begin{pmatrix}
\widetilde{g}_{22} & - \widetilde{g}_{12}\\
- \widetilde{g}_{12} & \widetilde{g}_{11}
\end{pmatrix}.
\end{align*}
Then
\begin{align*}
\mathring{g}_1 = \begin{pmatrix}
\cos \theta + \mathfrak{b}' (Y_1) \sin \theta \\
- \sin \theta + \mathfrak{b}' (Y_1) \cos \theta
\end{pmatrix},{ \ }
\mathring{g}_2 = \begin{pmatrix}
\sin \theta \\
\cos \theta
\end{pmatrix},{ \ }
\mathring{g}_{11} = 1 + (\mathfrak{b}')^2 ,{ \ }\mathring{g}_{22} = 1,\\
{ \ }\mathring{g}_{12} =\mathring{g}_{21} = \mathfrak{b}',{ \ }\mathring{\mathcal{G}} = 1,{ \ }\mathring{g}^{11} = 1 ,{ \ }\mathring{g}^{22} = 1 + (\mathfrak{b}')^2,{ \ }\mathring{g}^{12} = \mathring{g}^{21} = - \mathfrak{b}',
\end{align*}
\begin{align*}
\widetilde{g}_{11} & =  (\cos \theta + \mathfrak{b}'\sin \theta)^2 \tilde{\mathfrak{g}}_{11} + (- \sin \theta + \mathfrak{b}' \cos \theta)^2 \tilde{\mathfrak{g}}_{22} + C_\sharp \tilde{\mathfrak{g}}_{12},\\
\widetilde{g}_{22} & = \sin^2 \theta \tilde{\mathfrak{g}}_{11} + \cos^2 \theta \tilde{\mathfrak{g}}_{22} + 2 \sin \theta \cos \theta \tilde{\mathfrak{g}}_{12},\\
\widetilde{g}_{12} & =  ( \cos \theta \sin \theta + \mathfrak{b}' \sin^2 \theta) \tilde{\mathfrak{g}}_{11} + ( - \sin \theta \cos \theta + \mathfrak{b}' \cos^2 \theta ) \tilde{\mathfrak{g}}_{22} + C_{\sharp \sharp} \tilde{\mathfrak{g}}_{12},\\
\widetilde{g}^{22} & = (- \sin \theta + \mathfrak{b}' \cos \theta)^2 \tilde{\mathfrak{g}}^{11} + (\cos \theta + \mathfrak{b}'\sin \theta)^2 \tilde{\mathfrak{g}}^{22} - C_\sharp \tilde{\mathfrak{g}}^{12},\\
\widetilde{g}^{11} & =  \cos^2 \theta \tilde{\mathfrak{g}}^{11} + \sin^2 \theta \tilde{\mathfrak{g}}^{22} - 2 \sin \theta \cos \theta \tilde{\mathfrak{g}}^{12},\\
\widetilde{g}^{12} & = ( - \sin \theta \cos \theta + \mathfrak{b}' \cos^2 \theta ) \tilde{\mathfrak{g}}^{11} + ( \cos \theta \sin \theta + \mathfrak{b}' \sin^2 \theta) \tilde{\mathfrak{g}}^{22} - C_{\sharp \sharp} \tilde{\mathfrak{g}}^{12},
\end{align*}
$\widetilde{g}_{21} = \widetilde{g}_{12}$, $\widetilde{g}^{21} = \widetilde{g}^{12}$, and $\widetilde{\mathcal{G}} = \tilde{\mathfrak{g}}_{11} \tilde{\mathfrak{g}}_{22} - \tilde{\mathfrak{g}}_{12} \tilde{\mathfrak{g}}_{21}$, where $C_\sharp := 2(\cos \theta + \mathfrak{b}'\sin \theta)(- \sin \theta + \mathfrak{b}'  \cos \theta)$ and $C_{\sharp \sharp} := \cos^2 \theta - \sin^2 \theta + 2\mathfrak{b}' \sin \theta \cos \theta$. Moreover, for all $\xi = { }^t ( \xi_1 , \xi_2 ) \in \mathbb{R}^2$
\begin{equation}\label{EQ53}
\mathring{g}^{\alpha \beta} \xi_\alpha \xi_\beta \geq C_{\mathfrak{b}} | \xi |^2,
\end{equation}
where $C_{\mathfrak{b}} : = \min \{ 1/\{1 +2 (\mathfrak{b}')^2\}, 1/2\}$.

\end{lemma}

\begin{proof}[Proof of Lemma \ref{lem51}]
Fix $P = { }^t (P_1 , P_2) \in \partial U$. Since $\partial U$ is $C^2$-class, there are a local coordinate $Z= { }^t (Z_1,Z_2)$, $\delta>0$, $\mathfrak{b} \in C^2 (\mathbb{R})$, and $\theta \in \mathbb{R}$ such that
\begin{equation*}
B_\delta (P) \cap U = \{ Z = { }^t (Z_1 , Z_2 ) \in \mathbb{R}^2; { \ } - \delta < Z_1 < \delta , \mathfrak{b} (Z_1) < Z_2 < \sqrt{ \delta - Z_1^2 } \}
\end{equation*}
and
\begin{equation*}
Z = 
\begin{pmatrix}
Z_1\\
Z_2
\end{pmatrix}
=
\begin{pmatrix}
\cos \theta & - \sin \theta\\
\sin \theta & \cos \theta
\end{pmatrix}
\begin{pmatrix}
X_1 - P_1\\
X_2 - P_2
\end{pmatrix}.
\end{equation*}
Set $Y = { }^t (Y_1 , Y_2)$ by $Y_1 = Z_1$ and $Y_2 = Z_2 - \mathfrak{b} (Z_1)$. Then we have
\begin{equation*}
Q_+^{\delta , \mathfrak{b}} = \{ Y = { }^t (Y_1 , Y_2 ) \in \mathbb{R}^2; { \ }  -\delta < Y_1 < \delta,{ \ }0 < Y_2 < \sqrt{ \delta^2 - Y_1^2} - \mathfrak{b} (Y_1) \}.
\end{equation*}
Now we set
\begin{multline*}
\Phi = \Phi (X) =
\begin{pmatrix}
\Phi_1\\
\Phi_2
\end{pmatrix}
\\=
\begin{pmatrix}
( X_1 - P_1 ) \cos \theta - ( X_2 - P_2 )\sin \theta\\
( X_1 - P_1 ) \sin \theta + ( X_2 - P_2 ) \cos \theta - \mathfrak{b} ( (X_1 - P_1) \cos \theta - (X_2 - P_2) \sin \theta )
\end{pmatrix}.
\end{multline*}
It is clear that $\Phi : B_\delta (P) \cap U \to Q_+^{\delta , \mathfrak{b}}$ and $\Phi \in C^2 ( B_\delta (P) \cap U)$. Set
\begin{equation*}
\Phi^{-1} = \Phi^{-1} (Y) =
\begin{pmatrix}
\Phi_1^{-1}\\
\Phi_2^{-1}
\end{pmatrix}
=
\begin{pmatrix}
\cos \theta & \sin \theta\\
- \sin \theta & \cos \theta
\end{pmatrix}
\begin{pmatrix}
Y_1 \\
Y_2 + \mathfrak{b} (Y_1) 
\end{pmatrix} + 
\begin{pmatrix}
P_1\\
P_2
\end{pmatrix}.
\end{equation*}
It is easy to check that $\Phi^{-1}$ is the inverse function of $\Phi$ and that
\begin{align*}
& \Phi^{-1} : Q_+^{\delta , \mathfrak{b}} \to B_\delta (P),\\
& \Phi^{-1} \in C^2 (Q_+^{\delta , \mathfrak{b}}).
\end{align*}
Direct calculations give \eqref{EQ51} and \eqref{EQ52}. Therefore, the lemma follows.
\end{proof}
\begin{proof}[Proof of Lemma \ref{lem52}]
We only derive \eqref{EQ53}. Fix $\xi = { }^t ( \xi_1 , \xi_2 ) \in \mathbb{R}^2$. Since
\begin{align*}
2 \mathfrak{b}' \xi_1 \xi_2 & = 2 \left( \frac{\mathfrak{b}'}{ \sqrt{1/2 + ( \mathfrak{b}')^2} } \xi_1 \right) \left( \sqrt{1/2 + ( \mathfrak{b}')^2  } \xi_2 \right)\\
 & \leq \frac{ ( \mathfrak{b}')^2}{1/2 + ( \mathfrak{b}')^2 } \xi_1^2 + \left\{ \frac{1}{2} + ( \mathfrak{b}')^2 \right\} \xi_2^2,
\end{align*}
we see that
\begin{equation}\label{EQ54}
- 2 \mathfrak{b}' \xi_1 \xi_2 \geq - \frac{ ( \mathfrak{b}')^2}{1/2 + ( \mathfrak{b}')^2 } \xi_1^2 - \left\{ \frac{1}{2} + ( \mathfrak{b}')^2 \right\} \xi_2^2.
\end{equation}
Using \eqref{EQ54}, we find that for each $\xi \in \mathbb{R}^2$
\begin{align*}
\mathring{g}^{\alpha \beta} \xi_\alpha \xi_\beta & = \xi_1^2 + \{ 1 + (\mathfrak{b}')^2 \} \xi_2^2 - 2 \mathfrak{b}' \xi_1 \xi_2\\
& \geq \frac{ 1/2 }{1/2 + ( \mathfrak{b}')^2 } \xi_1^2 + \frac{1}{2} \xi_2^2\\
& \geq \min \left\{ \frac{ 1/2 }{1/2 + ( \mathfrak{b}')^2 },  \frac{1}{2} \right\} ( \xi_1^2 + \xi_2^2 ).
\end{align*}
Therefore, we have \eqref{EQ53}.
\end{proof}

Let us now attack Proposition \ref{prop29}.
\begin{proof}[Proof of Proposition \ref{prop29}]
Let $F \in L^2 (\Gamma_0)$ and $w \in W_0^{1,2} (\Gamma_0)$. Assume that
\begin{equation}\label{EQ55}
\dual{ \nabla_\Gamma w , \nabla_\Gamma \phi } = \dual{ F , \phi }
\end{equation}
holds for all $\phi \in W_0^{1,2} ( \Gamma_0 )$. We first show the following lemma.
\begin{lemma}\label{Lem53}
There is $C = C (\Gamma_0) >0$ such that
\begin{align}
\| w \|_{W^{1,2} ( \Gamma_0 )} & \leq C \| F \|_{L^2 (\Gamma_0)},\label{EQ56}\\
\| \widehat{w} \|_{W^{1,2} ( U )} & \leq C \| F \|_{L^2 (\Gamma_0)}\label{EQ57}.
\end{align}
\end{lemma}
\begin{proof}[Proof of Lemma \ref{Lem53}]
Since $w \in W_0^{1,2} (\Gamma_0)$, we apply \eqref{EQ55} to have
\begin{equation*}
\| \nabla_\Gamma w \|_{L^2 (\Gamma_0)}^2 = \dual{ F , w }.
\end{equation*}
Using the surface Poincar\'e inequality \eqref{eq331}, we obtain
\begin{equation*}
\| \nabla_\Gamma w \|_{L^2 (\Gamma_0)}^2 \leq C \| F \|_{L^2 (\Gamma_0)}^2 + \frac{1}{2} \| \nabla_\Gamma w \|_{L^2 (\Gamma_0 )}^2.
\end{equation*}
This gives
\begin{equation*}
\| \nabla_\Gamma w \|_{L^2 (\Gamma_0)}^2 \leq C \| F \|_{L^2 (\Gamma_0)}^2.
\end{equation*}
By the surface Poincar\'e inequality \eqref{eq331}, we have
\begin{equation*}
\| w \|_{W^{1,2} (\Gamma_0)} \leq C \| F \|_{L^2 (\Gamma_0)},
\end{equation*}
which is \eqref{EQ56}. From \eqref{eq320} and \eqref{EQ56}, we see \eqref{EQ57}. Therefore, the lemma follows.
\end{proof}
Now we return to the proof of Proposition \ref{prop29}. Write
\begin{equation*}
\mathcal{A} \varphi = - \frac{\partial}{\partial X_\alpha} \left( \sqrt{\mathcal{G}} \mathfrak{g}^{\alpha \beta} \frac{\partial \varphi }{\partial X_\beta} \right) .
\end{equation*}
Applying Lemmas \ref{Lem38} and \ref{Lem310}, and \eqref{EQ55}, we find that for all $\varphi \in C_0^2 (U)$
\begin{equation*}
\int_U (\mathcal{A} \widehat{w}) \varphi { \ } d X = \int_U \widehat{F} \varphi \sqrt{ \mathcal{G} } { \ }d X.
\end{equation*}
This implies that
\begin{equation}\label{EQ58}
\mathcal{A} \widehat{w} = \widehat{F} \sqrt{ \mathcal{G} } \text{ a.e. in }U.
\end{equation}

From Lemma \ref{lem51} we find that for each fixed $P \in \partial U$ there are $\delta = \delta (P) >0$, $\mathfrak{b} = \mathfrak{b}(P) \in C^2 (\mathbb{R})$, and $\Phi = \Phi (P) \in C^2 (B_\delta (P) \cap U )$ satisfying the properties as in Lemma \ref{lem51}. Since $\partial U \subset \bigcup_{P \in \partial U} B_{\delta(P)} (P)$ and $\partial U$ is a compact set, there are $\{ P^k \}_{k =1}^m \subset \partial U$ such that
\begin{equation*}
\partial U \subset \bigcup_{ k = 1}^m B_{\delta (P^k)} (P^k) .
\end{equation*}
Write $\Omega_k = B_{\delta (P^k)} (P^k)$. Take an open set $\Omega_0 \Subset U$ such that
\begin{equation*}
U \subset \bigcup_{ i =0}^m \Omega_i .
\end{equation*}
From the partition of unity, there are $\eta_i \in C_0^\infty ( \Omega_i )$ such that 
\begin{equation*}
\text{supp} { \ }\eta_i \subset \Omega_i \text{ and }\sum_{i=0}^m \eta_i = 1 \text{ in }U.
\end{equation*}
Set $\hat{w}_i = \eta_i \widehat{w}$. It is clear that $\widehat{w} = \sum_{i=0}^m \hat{w}_i$. By \eqref{EQ57}, we see that
\begin{equation}\label{EQ59}
\| \hat{w}_i \|_{W^{1,2}( \Omega_i \cap U )} \leq C \| F \|_{L^2 (\Gamma_0 )}.
\end{equation}
Applying \eqref{EQ58}, we find that
\begin{align*}
\mathcal{A} \hat{w}_i & = \mathcal{A} (\eta_i \widehat{w})\\
& = ( \mathcal{A} \eta_i ) \widehat{w} + \eta_i ( \mathcal{A} \widehat{w})\\
& =  ( \mathcal{A} \eta_i ) \widehat{w} + \eta_i ( \widehat{F} \sqrt{ \mathcal{G} }) := \hat{F}_i.
\end{align*}
Since $\eta_i \in C_0^\infty ( \Omega_i )$ and $\widehat{x} \in [C^2 (\overline{U})]^3$, we use \eqref{EQ57} to find that
\begin{equation}\label{EQ510}
\| \hat{F}_i \|_{L^2 (\Omega_i \cap U )} \leq C \| F \|_{L^2 (\Gamma_0)}.
\end{equation}
Now we consider
\begin{equation}\label{EQ511}
\mathcal{A} \hat{w}_i = \hat{F}_i \text{ a.e. in } \Omega_i \cap U.
\end{equation}
Fix $\varphi_i \in W_0^{1 , 2} ( \Omega_i \cap U)$. Using the integration by parts, we have
\begin{equation}\label{EQ512}
\int_{\Omega_i \cap U } \mathfrak{g}^{\alpha \beta} \frac{\partial \hat{w}_i}{\partial X_\alpha} \frac{\partial \varphi_i }{\partial X_\beta} \sqrt{\mathcal{G}} { \ }d X = \int_{\Omega_i \cap U} \hat{F}_i \varphi_i \sqrt{\mathcal{G}} { \ }d X. 
\end{equation}
Now we prove that for each $i \in \{0,1, \cdots, m \}$, $\hat{w}_i \in W^{2,2} ( \Omega_i \cap U )$ and  
\begin{equation*}
\| \hat{w}_i \|_{W^{2,2} (\Omega_i \cap U )} \leq C \| F \|_{L^2 (\Gamma_0)}.
\end{equation*}

We first consider $\hat{w}_1$. For simplicity we write $\delta (P_1)$ and $\mathfrak{b}(P_1)$ as $\delta$ and $\mathfrak{b}$. Let ${Q_+^{\delta, \mathfrak{b}}}$, $\mathring{g}_{\alpha \beta}$, $\mathring{g}^{\alpha \beta}$, $\widetilde{g}_{\alpha \beta}$, $\widetilde{g}^{\alpha \beta}$, $\Psi$, $\mathring{ \mathcal{G}}$, $\widetilde{ \mathcal{G}}$ be the symbols as in Lemmas \ref{lem51} and \ref{lem52}. Set
\begin{equation*}
\Gamma_1 = \{ x \in \mathbb{R}^3; { \ } x = \widehat{x} (X), X \in \Omega_1 \cap U \}.
\end{equation*}
From \eqref{EQ512}, we observe that for each $\varphi \in W_0^{1,2} ( \Omega_1 \cap U )$
\begin{equation}\label{EQ513}
\int_{\Omega_1 \cap U } \mathfrak{g}^{\alpha \beta} \frac{\partial \hat{w}_1}{\partial X_\alpha} \frac{\partial \varphi }{\partial X_\beta} \sqrt{\mathcal{G}} { \ }d X = \int_{\Omega_1 \cap U} \hat{F}_1 \varphi \sqrt{\mathcal{G}} { \ }d X. 
\end{equation}
Applying the change of variables, we find that
\begin{align*}\text{(L.H.S.) of }\eqref{EQ513} & = \int_{\Gamma_1} \nabla_\Gamma \breve{\hat{w}}_1 \cdot \nabla_\Gamma \breve{\varphi} { \ } d \mathcal{H}^2_x\\
& = \int_{Q_+^{\delta, \mathfrak{b}}} \widetilde{g}^{\alpha \beta} \frac{\partial \widetilde{w}_1}{\partial Y_\alpha} \frac{\partial \widetilde{\varphi}}{\partial Y_\beta} \sqrt{ \widetilde{ \mathcal{G} }} { \ } d Y
\end{align*}
and that
\begin{equation*}
\text{(R.H.S) of }\eqref{EQ513} = \int_{Q_+^{\delta, \mathfrak{b}}} \widetilde{F}_1 \widetilde{\varphi} \sqrt{ \widetilde{\mathcal{G}}} { \ } d Y.
\end{equation*}
Here $\breve{f}= f ( \breve{X} (x))$, $\widetilde{w}_1 = \widetilde{w}_1 (Y) := \hat{w} (\Psi (Y))$, $\widetilde{\varphi} = \widetilde{\varphi} (Y) := \varphi (\Psi (Y))$, and $\widetilde{F}_1 = \widetilde{F}_1 (Y) := \hat{F}_1 (\Psi (Y))$. Thus, we have
\begin{equation}\label{EQ514}
\int_{Q_+^{\delta, \mathfrak{b}}} \widetilde{g}^{\alpha \beta} \frac{\partial \widetilde{w}_1}{\partial Y_\alpha} \frac{\partial \widetilde{\varphi}}{\partial Y_\beta} \sqrt{ \widetilde{ \mathcal{G} }} { \ } d Y = \int_{Q_+^{\delta, \mathfrak{b}}} \widetilde{F}_1 \widetilde{\varphi} \sqrt{ \widetilde{\mathcal{G}}} { \ } d Y.
\end{equation}

Now we prove that there is $C_1 >0$ such that for all $\widetilde{\varphi} \in W^{1,2} (Q_+^{\delta, \mathfrak{b}} )$
\begin{equation}\label{EQ515}
\int_{Q_+^{\delta, \mathfrak{b}}} \widetilde{g}^{\alpha \beta} \frac{\partial \widetilde{\varphi}}{\partial Y_\alpha} \frac{\partial \widetilde{\varphi}}{\partial Y_\beta} \sqrt{ \widetilde{ \mathcal{G} }} { \ } d Y \geq C_1 \int_{Q_+^{\delta, \mathfrak{b}}} | \nabla_Y \widetilde{\varphi} |^2 { \ } d Y.
\end{equation}
Using the change of variables and Lemmas \ref{Lem31}, \ref{lem51}, and \ref{lem52}
, we check that
\begin{align}\label{EQ516}
\int_{Q_+^{\delta, \mathfrak{b}}} \widetilde{g}^{\alpha \beta} \frac{\partial \widetilde{ \varphi }}{\partial Y_\alpha} \frac{\partial \widetilde{\varphi}}{\partial Y_\beta} \sqrt{ \widetilde{ \mathcal{G} }} { \ } d Y & = \int_{\Gamma_1} | \nabla_\Gamma \breve{\varphi} |^2 { \ }  d \mathcal{H}^2_x\\
& = \int_{\Omega_1 \cap U } \mathfrak{g}^{\alpha \beta} \frac{\partial \varphi}{\partial X_\alpha} \frac{\partial \varphi }{\partial X_\alpha} \sqrt{\mathcal{G}} { \ }d X.\notag
\end{align}
By Definition \ref{def21} and Assumption \ref{ass24}, we see that
\begin{equation}\label{EQ517}
\int_{\Omega_1 \cap U } \mathfrak{g}^{\alpha \beta} \frac{\partial \varphi}{\partial X_\alpha} \frac{\partial \varphi }{\partial X_\alpha} \sqrt{\mathcal{G}} { \ }d X \geq \lambda_0 \lambda_{min} \int_{\Omega_1 \cap U} |\nabla_X \varphi |^2 { \ } d X.
\end{equation}
Using the change of variables, Lemmas \ref{lem51}, \ref{lem52}, and \eqref{EQ53}, we find that
\begin{align}\label{EQ518}
\int_{\Omega_1 \cap U} |\nabla_X \varphi |^2 { \ } d X & = \int_{Q_+^{\delta, \mathfrak{b}}} \mathring{g}^{\alpha \beta} \frac{\partial \widetilde{\varphi} }{\partial Y_\alpha} \frac{\partial \widetilde{\varphi} }{\partial Y_\beta} { \ } d Y\\
& \geq C_{\mathfrak{b}} \int_{Q_+^{\delta, \mathfrak{b}}} | \nabla_Y \widetilde{\varphi} |^2 { \ } d Y.\notag
\end{align}
Combining \eqref{EQ516}-\eqref{EQ518}, we have \eqref{EQ515}.

Next we introduce the difference quotient $D_h$. For $0 < h <<1$,
\begin{equation*}
D_h f := \frac{f (Y + h e_1) - f (Y)}{h }.
\end{equation*}
Here $e_1 = { }^t (1,0)$. See \cite[\S 5.8]{Eva10} for some properties of $D_h$. Since $\hat{w}_1 = \eta_1 \widehat{w}$, $\widehat{w} \in W_0^{1,2} (U)$, and supp $\eta_1 \Subset \Omega_1 $, we see that $\hat{w}_1 \in W_0^{1,2} ( \Omega_1 \cap U)$. From the argument in the proof of Lemma \ref{lem51} and supp $\eta_1 \Subset \Omega_1 $, we check that $\widetilde{w}_1 \in W^{1,2}_0 ( {Q_+^{\delta, \mathfrak{b}}})$ and that $D_h \widetilde{w}_1,D_{-h} \widetilde{w}_1 ,D_{-h} D_h \widetilde{w}_1 \in W^{1,2}_0 ( {Q_+^{\delta, \mathfrak{b}}})$. Using \eqref{EQ514}, we have
\begin{equation}\label{EQ519}
\int_{Q_+^{\delta, \mathfrak{b}}} \widetilde{g}^{\alpha \beta} \frac{\partial \widetilde{w}_1 }{\partial Y_\alpha} \left( \frac{\partial  }{\partial Y_\beta} D_{-h} D_h \widetilde{w}_1 \right) \sqrt{\widetilde{\mathcal{G}}} { \ } d Y = \int_{Q_+^{\delta, \mathfrak{b}}} \widetilde{F}_1 ( D_{-h} D_h \widetilde{w}_1) \sqrt{\widetilde{\mathcal{G}}} { \ } d Y.
\end{equation}
By the definition of $D_h$, we check that
\begin{equation*}
\text{(L.H.S) of }\eqref{EQ519} = \int_{Q_+^{\delta, \mathfrak{b}}} D_h \left( \sqrt{\widetilde{\mathcal{G}}} \widetilde{g}^{\alpha \beta} \frac{\partial \widetilde{w}_1 }{\partial Y_\alpha} \right) \left( \frac{\partial  }{\partial Y_\beta} D_h \widetilde{w}_1 \right) { \ } d Y = \mathcal{K}_1 + \mathcal{K}_2.
\end{equation*}
Here
\begin{align*}
\mathcal{K}_1 &:= \int_{Q_+^{\delta, \mathfrak{b}}} D_h \left( \sqrt{\widetilde{\mathcal{G}}} \widetilde{g}^{\alpha \beta} \right) \frac{\partial \widetilde{w}_1 }{\partial Y_\alpha} [Y_1 + h,Y_2] \left( \frac{\partial  }{\partial Y_\beta} D_h \widetilde{w}_1 \right) { \ } d Y,\\
\mathcal{K}_2 & := \int_{Q_+^{\delta, \mathfrak{b}}} \sqrt{\widetilde{\mathcal{G}}} \widetilde{g}^{\alpha \beta} \left( \frac{\partial  }{\partial Y_\alpha} D_h \widetilde{w}_1 \right) \left( \frac{\partial  }{\partial Y_\beta} D_h \widetilde{w}_1 \right) { \ } d Y. 
\end{align*}
As a result, we have
\begin{equation}\label{EQ520}
\mathcal{K}_2 = - \mathcal{K}_1 + \int_{Q_+^{\delta, \mathfrak{b}}} \widetilde{F}_1 ( D_{-h} D_h \widetilde{w}_1) \sqrt{\widetilde{\mathcal{G}}} { \ } d Y.
\end{equation}
By \eqref{EQ515}, we find that
\begin{equation}\label{EQ521}
\mathcal{K}_2 \geq C_1 \int_{Q_+^{\delta, \mathfrak{b}}} | \nabla_Y D_h \widetilde{w}_1 |^2 { \ } d Y. 
\end{equation}
Applying the Cauchy inequality, we see that
\begin{equation*}
- \mathcal{K}_1 \leq C \int_{Q_+^{\delta, \mathfrak{b}}}  | \nabla_Y \widetilde{w}_1 |^2 { \ } d Y + \frac{C_1}{4} \int_{Q_+^{\delta, \mathfrak{b}}} | \nabla_Y D_h \widetilde{w}_1 |^2 { \ } d Y.
\end{equation*}
Since supp $D_{-h} D_h \widetilde{w}_1 \Subset {Q_+^{\delta, \mathfrak{b}}}$ for each fixed $Y_2$, we use the Cauchy inequality to check that
\begin{equation*}
\int_{Q_+^{\delta, \mathfrak{b}}} \widetilde{F}_1 ( D_{-h} D_h \widetilde{w}_1) \sqrt{\widetilde{\mathcal{G}}} { \ } d Y \leq C \int_{Q_+^{\delta, \mathfrak{b}}} | \widetilde{F}_1 |^2 { \ } d Y + \frac{C_1}{4} \int_{Q_+^{\delta, \mathfrak{b}}} | D_h \nabla_Y \widetilde{w}_1 |^2 { \ } d Y.
\end{equation*}
Here we used the fact that
\begin{align*}
\int_{Q_+^{\delta, \mathfrak{b}}} | D_{-h} D_h \widetilde{w}_1 |^2 { \ } d Y & \leq C \int_{Q_+^{\delta, \mathfrak{b}}} | \partial_{Y_1} D_h \widetilde{w}_1 |^2 { \ } d Y\\
& \leq C \int_{Q_+^{\delta, \mathfrak{b}}} | \nabla_Y D_h \widetilde{w}_1 |^2 { \ } d Y.
\end{align*}
Thus, we have
\begin{equation*}
\text{(R.H.S) of }\eqref{EQ520} \leq C \int_{Q_+^{\delta, \mathfrak{b}}} \{ | \widetilde{F}_1 |^2 + | \nabla_Y \widetilde{w}_1 |^2 \} { \ } d Y + \frac{C_1}{2} \int_{Q_+^{\delta, \mathfrak{b}}} | \nabla_Y D_h \widetilde{w}_1 |^2 { \ } d Y. 
\end{equation*}
From \eqref{EQ521}, we obtain
\begin{equation*}
\int_{Q_+^{\delta, \mathfrak{b}}} | \nabla_Y D_h \widetilde{w}_1 |^2 { \ } d Y \leq C \int_{Q_+^{\delta, \mathfrak{b}}} \{ | \widetilde{F}_1 |^2 + | \nabla_Y \widetilde{w}_1 |^2 \} { \ } d Y. 
\end{equation*}
Using \eqref{EQ59}, \eqref{EQ510}, and $\nabla_Y D_h \widetilde{w}_1 = D_h \nabla_Y \widetilde{w}_1$, we see that
\begin{equation*}
\int_{Q_+^{\delta, \mathfrak{b}}} | D_h \nabla_Y \widetilde{w}_1 |^2 { \ } d Y \leq C \| F \|_{L^2 (\Gamma_0)}. 
\end{equation*}
From the nice property of $D_h$ (\cite[\S 5.8]{Eva10}), we find that
\begin{equation*}
\int_{Q_+^{\delta, \mathfrak{b}}} | \partial_{Y_1} \nabla_Y \widetilde{w}_1 |^2 { \ } d Y \leq C \| F \|_{L^2 (\Gamma_0)}. 
\end{equation*}
This shows that $\partial_{Y_1}^2 \widetilde{w}_1, \partial_{Y_1}\partial_{Y_2} \widetilde{w}_{1} $, $\partial_{Y_2} \partial_{Y_1} \widetilde{w}_{1} $ are in $L^2 (Q_+^{\delta , \mathfrak{b}})$. From \eqref{EQ514}, we see that
\begin{equation*}
- \frac{\partial}{\partial Y_\alpha} \left( \sqrt{\mathcal{\widetilde{\mathcal{G}}}} \widetilde{g}^{\alpha \beta} \frac{\partial \widetilde{w}_1 }{\partial Y_\beta} \right) = \widetilde{F}_1 \sqrt{\widetilde{\mathcal{G}}} \text{ a.e. in } Q_+^{\delta, \mathfrak{b}} .
\end{equation*}
Since 
\begin{multline*}
 (\partial_{Y_2})^2 \widetilde{w}_1 = \frac{1}{\sqrt{ \widetilde{\mathcal{G}}} \widetilde{g}^{22} } \bigg( - \widetilde{F}_1 \sqrt{\widetilde{\mathcal{G}}} - \frac{\partial}{\partial X_1} (\sqrt{\widetilde{\mathcal{G}}} \widetilde{g}^{1\alpha} \partial_{Y_\alpha} \widetilde{w}_1  ) \\
- \frac{\partial}{\partial X_2} (\sqrt{\widetilde{\mathcal{G}}} \widetilde{g}^{12} \partial_{Y_1} \widetilde{w}_1  ) - \frac{\partial}{\partial X_2} (\sqrt{\widetilde{\mathcal{G}}} \widetilde{g}^{22} )\partial_{Y_2} \widetilde{w}_1 \bigg),
\end{multline*}
we find that $(\partial_{Y_2})^2 \widetilde{w}_1 \in L^2 (Q_+^{\delta, \mathfrak{b}})$ and that
\begin{equation*}
\| (\partial_{Y_2})^2 \widetilde{w}_1 \|_{L^2 ( Q_+^{\delta , \mathfrak{b}})} \leq C \| F \|_{L^2 (\Gamma_0)}.
\end{equation*}
Therefore, we see that $\widetilde{w}_1 \in W^{2,2} ( Q_+^{\delta, \mathfrak{b}})$ and that
\begin{equation*}
\| \widetilde{w}_1 \|_{W^{2,2} ( Q_+^{\delta , \mathfrak{b}})} \leq C \| F \|_{L^2 (\Gamma_0)}.
\end{equation*}
Using the change of variables, we conclude that $\hat{w}_1 \in W^{2,2}( \Omega_1 \cap U )$ and that
\begin{equation}\label{EQ522}
\| \hat{w}_1 \|_{W^{2,2} ( \Omega_1 \cap U )} \leq C \| F \|_{L^2 (\Gamma_0)}.
\end{equation}
Similarly, we see that for each $i=2, \cdots, m$, $\hat{w}_i \in W^{2,2}( \Omega_i \cap U )$ and that
\begin{equation}\label{EQ523}
\| \hat{w}_i \|_{W^{2,2} ( \Omega_i \cap U )} \leq C \| F \|_{L^2 (\Gamma_0)}.
\end{equation}

Next we consider $\hat{w}_0$. Using \eqref{EQ511}, we see that for all $\varphi \in W_0^{1,2} (\Omega_0)$
\begin{equation}\label{EQ524}
\int_{\Omega_0 } \mathfrak{g}^{\alpha \beta} \frac{\partial \hat{w}_0}{\partial X_\alpha} \frac{\partial \varphi}{\partial X_\beta} \sqrt{ \mathcal{G}} { \ }d X = \int_{\Omega_0} \hat{F}_0 \varphi \sqrt{\mathcal{G}} { \ } d X.
\end{equation}
By Definition \ref{def21} and Assumption \ref{ass24}, we see that for all $\varphi \in W_0^{1,2} (\Omega_0)$
\begin{equation}\label{EQ525}
\int_{U} \mathfrak{g}^{\alpha \beta} \frac{\partial \varphi}{\partial X_\alpha} \frac{\partial \varphi}{\partial X_\beta} \sqrt{ \mathcal{G}} { \ }d X \geq C \int_{\Omega_0} | \nabla_X \varphi |^2 { \ }d X.
\end{equation}
Since supp $\hat{w}_0 \Subset \Omega_0$, we see that $D_h \hat{w}_0 , D_{-h} \hat{w}_0, D_{-h}D_h \hat{w}_0 \in W^{1,2}_0 ( \Omega_0)$ if $h$ is sufficently small. From \eqref{EQ524}, we have
\begin{equation*}
\int_{\Omega_0 } \mathfrak{g}^{\alpha \beta} \frac{\partial \hat{w}_0}{\partial X_\alpha} \left( \frac{\partial}{\partial X_\beta} D_{-h} D_h \hat{w}_0 \right) \sqrt{ \mathcal{G}} { \ }d X = \int_{\Omega_0} \hat{F}_0 (D_{-h}D_h \hat{w}_0) \sqrt{\mathcal{G}} { \ } d X.
\end{equation*}
By the previous argument with \eqref{EQ525}, we see that $\hat{w}_0 \in W^{2,2}( \Omega_0 )$ and that
\begin{equation}\label{EQ526}
\| \hat{w}_0 \|_{W^{2,2} ( \Omega_0 )} \leq C \| F \|_{L^2 (\Gamma_0)}.
\end{equation}

Using \eqref{EQ522}, \eqref{EQ523}, and \eqref{EQ526}, we check that
\begin{equation*}
\| \widehat{w} \|_{W^{2,2} (U)} \leq \sum_{i=0}^m \| \hat{w}_i \|_{W^{2,2} (U)} \leq C \| F \|_{L^2 (\Gamma_0 )}.
\end{equation*}
This implies that $\widehat{w} \in W^{2,2} (U)$. We also see that $w \in W^{2,2} (\Gamma_0)$ and
\begin{align*}
\| w \|_{W^{2,2} (\Gamma ) } \leq C \| \widehat{w} \|_{W^{2,2} (U)} \leq C \| F \|_{L^2 (\Gamma_0 )}.
\end{align*}
Therefore, Proposition \ref{prop29} is proved.
\end{proof}

\section{Existence of strong solutions}\label{sect6}
We prove Theorems \ref{thm25}-\ref{thm27} to show the existence of strong solutions to systems \eqref{eq11} and \eqref{eq12}. Let us first attack Theorem \ref{thm25}
\begin{proof}[Proof of Theorem \ref{thm25}]
Let $F \in L_0^2 ( \Gamma_0 )$. From Propositions \ref{prop28}, there exists $v \in H^1 (\Gamma_0)$ such that for all $\psi \in W^{1,2} ( \Gamma_0)$
\begin{equation}\label{EQ61}
\dual{\nabla_\Gamma v , \nabla_\Gamma \psi } = \dual{ F , \psi }.
\end{equation}
Using Lemmas \ref{Lem38} and \ref{Lem310}, we see that for all $\varphi \in C_0^2 (U)$
\begin{align*}
\int_U \widehat{F} \varphi \sqrt{\mathcal{G}} { \ }d X & = \int_U \mathfrak{g}^{\alpha \beta} \frac{\partial \widehat{v}}{\partial X_\alpha} \frac{\partial \varphi}{\partial X_\beta} \sqrt{ \mathcal{G}} { \ } d X\\
& = - \int_U \left\{ \frac{1}{\sqrt{\mathcal{G}}} \frac{\partial}{\partial X_\alpha} \left( \sqrt{\mathcal{G}} \mathfrak{g}^{\alpha \beta} \frac{\partial \widehat{v} }{\partial X_\beta} \right) \right\} \sqrt{\mathcal{G}} \varphi { \ } d X.
\end{align*}
This implies that
\begin{equation*}
- \frac{\partial}{\partial X_\alpha} \left( \sqrt{\mathcal{G}} \mathfrak{g}^{\alpha \beta} \frac{\partial \widehat{v} }{\partial X_\beta} \right)  = \widehat{F}\sqrt{ \mathcal{G}} \text{ a.e. in } U.
\end{equation*}
Thus, we find that
\begin{equation*}
- \Delta_\Gamma v = F \text{ a.e. on }\Gamma_0.
\end{equation*}
Since $v \in W^{1,2} (\Gamma_0)$, it follows from Lemma \ref{Lem311} to see that there exists $u \in W^{2,2} ( \Gamma_0 )$ such that
\begin{equation*}
\gamma_2 u = \gamma_2 v.
\end{equation*}
Here $\gamma_2 : W^{1,2} (\Gamma_0) \to L^2 ( \partial \Gamma_0)$ is the trace operator. Set $w = v - u $ and $F_* = F - \Delta_\Gamma u$. It is easy to check that $w \in W^{1,2}_0 (\Gamma )$ and $F_* \in L^2 ( \Gamma_0 )$, and that $w$ satisfies
\begin{equation*}
\begin{cases}
- \Delta_\Gamma w =  F_* \text{ a.e. on }\Gamma_0,\\
\gamma_2 w =0.
\end{cases}
\end{equation*}
Using the integration by parts, we see that
\begin{equation*}
\dual{ \nabla_\Gamma w , \nabla_\Gamma \phi } = \dual{F_* , \phi }
\end{equation*}
holds for all $\phi \in W_0^{1,2} (\Gamma_0)$. Since $w \in W_0^{1,2} (\Gamma_0)$, it follows from Proposition \ref{prop29} to see that $w \in W^{2,2} ( \Gamma_0)$. From $u \in W^{2,2} ( \Gamma_0 )$, we find that $v \in W^{2,2} ( \Gamma_0 )$.

Applying Lemmas \ref{Lem38} and \ref{Lem310}, and \eqref{EQ61}, we observe that for all $\psi \in W^{1,2} (\Gamma_0)$
\begin{align*}
\dual{F , \psi} = \dual{- \Delta_\Gamma v , \psi } &= \dual{\nabla_\Gamma v , \nabla_\Gamma \psi }\\
 &= \dual{ F , \psi } + \int_{\partial \Gamma_0} \left( \gamma_2 \left[ \frac{\partial v}{\partial \nu} \right] \right) ( \gamma_2 \psi ) { \ }d \mathcal{H}^1_x.
\end{align*}
Since
\begin{equation*}
\int_{\partial \Gamma_0} \left( \gamma_2 \left[ \frac{\partial v}{\partial \nu} \right] \right) ( \gamma_2 \psi ) { \ }d \mathcal{H}^1_x = 0
\end{equation*}
for all $\psi \in W^{1,2} (\Gamma_0)$, we check that
\begin{equation*}
\left\| \gamma_2 \left[ \frac{\partial v}{\partial \nu} \right] \right\|_{L^2 (\partial \Gamma_0 )} = 0.
\end{equation*}
Therefore, we conclude that $v$ is a strong $L^2$-solution to system \eqref{eq11}. From Proposition \ref{prop28} we see the uniqueness of the strong $L^2$-solutions to \eqref{eq11}. Therefore, Theorem \ref{thm25} is proved.
\end{proof}

Next we show the existence of a strong $L^p$-solution to system \eqref{eq11}.
\begin{proof}[Proof of Theorem \ref{thm26}]
Fix $2 < p < \infty$ and $F \in L_0^p ( \Gamma_0)$. Since $L_0^p (\Gamma_0) \subset L_0^2 ( \Gamma_0 )$, it follows from Theorem \ref{thm25} to see that there exists a unique function $v \in H^1 (\Gamma_0) \cap W^{2,2} (\Gamma_0) $ such that
\begin{align}
\| \Delta_\Gamma v + F \|_{L^2 (\Gamma_0)} & = 0,\notag\\
\left\| \gamma_2 \left[ \frac{\partial v}{\partial \nu} \right] \right\|_{L^2 (\partial \Gamma_0)} & = 0.\label{EQ62}
\end{align}
From the arguments in the proofs of Proposition \ref{prop29} and Theorem \ref{thm25} we see that $\widehat{v}$ satisfies
\begin{equation*}
- \frac{\partial }{\partial X_\alpha} \left( \sqrt{\mathcal{G}} \mathfrak{g}^{\alpha \beta} \frac{\partial \widehat{v}}{\partial X_\beta} \right) = \widehat{F} \sqrt{\mathcal{G}} \text{ in }L^2(U).
\end{equation*}
Write
\begin{equation*}
\mathcal{A} \varphi := \frac{\partial }{\partial X_\alpha} \left( \sqrt{\mathcal{G}} \mathfrak{g}^{\alpha \beta} \frac{\partial \varphi}{\partial X_\beta} \right) .
\end{equation*}
Since $\widehat{v} \in W^{2,2} (U)$, it follows from the Sobolev embedding theorem to find that $\widehat{v} \in W^{1,p} (U)$. From $F \in L^p (\Gamma_0)$, we see that $\widehat{F} \sqrt{\mathcal{G}} \in L^p(U)$. By the argument in the proof of Lemma \ref{Lem311}, we find that there exists $V \in W^{2,p} (U)$ such that $\widehat{\gamma}_p V = \widehat{\gamma}_p \widehat{v}$. Here $\widehat{\gamma}_p: W^{1,p}(U) \to L^p(\partial U)$ is the trace operator. Set $\widehat{w} = \widehat{v} - V$. Then $\widehat{w}$ satisfies $\widehat{w} \in W_0^{1,p} (U)$ and
\begin{equation*}
- \mathcal{A} \widehat{w} = \widehat{F} \sqrt{ \mathcal{G}} + \mathcal{A} V \text{ in } L^2(U). 
\end{equation*}
This gives
\begin{equation*}
\mathfrak{g}^{\alpha \beta} \frac{\partial^2 \widehat{w} }{\partial X_\alpha \partial X_\beta} = \frac{1}{\sqrt{ \mathcal{G}}} \left\{ \widehat{F} \sqrt{ \mathcal{G}} - \mathcal{A} V + \frac{\partial \widehat{w}}{\partial X_\alpha} \left( \frac{\partial}{\partial X_\alpha}\sqrt{\mathcal{G}} \mathfrak{g}^{\alpha \beta}  \right) \right\} := F_\star.
\end{equation*}
It is easy to check that $F_\star \in L^p (U)$. Now we set
\begin{equation*}
\mathcal{L} \varphi := \mathfrak{g}^{\alpha \beta} \frac{\partial^2 \varphi}{\partial X_\alpha \partial X_\beta}.
\end{equation*}
By Assumption \ref{ass24}, we see that the operator $\mathcal{L}$ is strictly elliptic. Since $\widehat{w} \in W_0^{1,p}(U)$, $F_\star \in L^p (U)$, and $\mathcal{L}$ is strictly elliptic, it follows from \cite[Theorem 9.15]{GT98} to see that $\widehat{w} \in W^{2,p}(U)$. From $V \in W^{2,p} (U)$, we find that $\widehat{v} \in W^{2,p}(U)$. Therefore, we see that $v \in W^{2,p} (\Gamma_0)$. Since $W^{2,p}(U) \subset C^{1,1- 2/p}(\overline{U})$ from the Sobolev embedding theorem, we also see that $v \in C^{1, 1- 2/p} (\overline{\Gamma_0})$. 

Now we assume that $\partial U$ is $C^3$-class, $\widehat{x} \in [ C^3 ( \overline{U}) ]^3$, and that $F \in W^{1,p} (\Gamma_0)$. From \cite[Theorem 9.19]{GT98} and Sobolev embedding theorem, we find that $\widehat{w} \in W^{3,p}(U) \subset C^{2, 1- 2/p} ( \overline{U})$. Therefore we see that $v \in C^{2, 1 - 2/p } (\overline{\Gamma_0})$. Therefore, Theorem \ref{thm26} is proved.
\end{proof}

Finally, we prove Theorem \ref{thm27}.
\begin{proof}[Proof of Theorem \ref{thm27}]
We only show the assertion $(\mathrm{i})$ since $(\mathrm{ii})$ is similar. Let $2 \leq p < \infty$, and let $F \in L_0^p ( \Gamma_0)$ and $\chi \in W^{1,p} (\Gamma_0)$. To solve system \eqref{eq12}, we consider
\begin{equation}\label{EQ63}
\begin{cases}
- \Delta_\Gamma v = F + \chi H_\Gamma \text{ on } \Gamma_0,\\
\displaystyle{\frac{\partial v}{\partial \nu} = 0} \text{ on } \partial \Gamma_0.
\end{cases}
\end{equation}
From ${\rm{div}}_\Gamma ( \chi n ) = - \chi H_\Gamma$ on $\Gamma_0$ and $\chi n \cdot \nu = 0$ on $\partial \Gamma_0$, we apply the surface divergence theorem to see that
\begin{equation*}
\int_{\Gamma_0} \chi H_\Gamma { \ }d \mathcal{H}^2_x = 0.
\end{equation*}
This implies that $F + \chi H_\Gamma \in L^p_0 ( \Gamma_0 )$. From Theorem \ref{thm26}, there exists a solution $v \in W^{2,p} ( \Gamma_0)$ to system \eqref{EQ63}. Set $V = - \nabla_\Gamma v + \chi n$. It is easy to check that $V \in [W^{1,p}(\Gamma_0) ]^3$, ${\rm{div}}_\Gamma V = - \Delta_\Gamma v - \chi H_\Gamma n = F$, $V \cdot n = \chi$, and $ V \cdot \nu |_{\partial \Gamma_0} = - {\partial v}/{\partial \nu}|_{\partial \Gamma_0} = 0$. Therefore, Theorem \ref{thm27} is proved.
\end{proof}

\end{document}